\documentclass[10pt, reqno]{amsart}
\usepackage{graphicx} 
\usepackage{amsthm}
\usepackage{amsmath}
\usepackage{amssymb}
\usepackage{amsfonts}
\usepackage[left=2.5cm, right=2.5cm,top=2.5cm,bottom=2cm]{geometry}
\usepackage{bbm}
\usepackage{xcolor}
\usepackage{mathtools}
\usepackage{hyperref}
\hypersetup{colorlinks=true,linkcolor=blue,citecolor=red}
\usepackage{enumerate}
\usepackage{tikz}
\usepackage{tikzit}

\tikzstyle{Dotty}=[fill=black, draw=black, shape=circle, tikzit shape=circle]
\tikzstyle{gyyu}=[fill=cyan, draw=black, shape=circle, tikzit shape=circle]
\tikzstyle{doty}=[fill=black, draw=black, shape=circle, tikzit fill=black, tikzit draw=black, tikzit shape=circle]

\tikzstyle{eddt}=[fill=none, draw=black, tikzit draw=black, ->]
\tikzstyle{new edge style 0}=[-, draw=red, tikzit draw=red]
\tikzstyle{ArrowBlack}=[draw=none, tikzit draw=none, ->, fill=white]

\newtheorem{innercustomthm}{Theorem}
\newenvironment{customthm}[1]
{\renewcommand\theinnercustomthm{#1}\innercustomthm}
{\endinnercustomthm}
\newtheorem{theorem}{Theorem}[section]
\newtheorem{proposition}{Proposition} [section]
\newtheorem{definition}{Definition}[section]
\newtheorem{lemma}{Lemma}[section]

\numberwithin{equation}{section}
\newtheorem{remark}{Remark}[section]
\author{Sakil Ahamed}
\address{Sakil Ahamed, Debanjit Mondal \newline \indent Department of Mathematics \& Statistics, Indian Institute of Science Education and Research Kolkata, \newline \indent 
	Mohanpur, Nadia, West Bengal- 741246, India.}
\email{sakilahamed777@gmail.com, wrmarit@gmail.com}
\author{Debanjit Mondal}

\begin{document}

\title[GLOBAL CONTROLLABILITY OF THE KAWAHARA EQUATION]{GLOBAL CONTROLLABILITY OF THE KAWAHARA EQUATION AT ANY TIME}

\begin{abstract}  
	In this article, we prove that the nonlinear Kawahara equation on the periodic domain \(\mathbb{T}\) (the unit circle in the plane) is globally approximately controllable in \(H^s(\mathbb{T})\) for \(s \in \mathbb{N}\), at any time \(T > 0\), using a two-dimensional control force. The proof is based on the Agrachev-Sarychev approach in geometric control theory.  
\end{abstract}

\keywords{ Kawahara equation, Approximate controllability, Bourgain space, Geometric control theory. }
\subjclass{ 93B05, 35Q35, 35Q53. }
\allowdisplaybreaks
\maketitle


\section{Introduction}
The Kawahara equation \cite{Hunter_1988,kawahara_1972},
\begin{equation}\label{equkaw}
	u_t + \alpha u_{xxxxx} + \beta u_{xxx} + 3\gamma u u_x = 0,
\end{equation}
where \(\alpha, \beta, \gamma \in \mathbb{R}\) and \(\alpha, \gamma \neq 0\), generalizes the KdV equation by incorporating the fifth-order dispersive term. Kawahara \cite{kawahara_1972} originally studied this equation to describe solitary waves in scenarios where third-order dispersion is insufficient. Like the KdV equation, the Kawahara equation supports compressive or rarefactive solitary waves, determined by the sign of the dispersive terms. The equation has applications in modeling surface tension effects in shallow water and in describing critical magnetic-acoustic wave behavior in plasmas. It is also commonly referred to as the fifth-order KdV equation in the literature.

Since the late 1980s, the control theory of nonlinear dispersive wave equations has attracted considerable attention, fueled by rapid advances in their mathematical theory. The development of new analytical tools has enabled researchers to address previously intractable problems. Notably, the control theory of the KdV equation has been extensively studied, building on its mathematical progress, with significant contributions from various researchers (see \cite{Coron_2004,Chen_2023,Rosier_10,Rosier_Zhang_2009,Russell_Zhang_1996,Russell_Zhang_1993, Zang_1999}). The authors recommend the survey article \cite{Cerpa_2014}   and the references therein  for a thorough review of the well-posedness and controllability results of the KdV equation.

In contrast, the control theory of the Kawahara equation has received relatively limited attention, with only a few studies exploring this area (cf. \cite{Zhang_Zhao_2012,Zhang_Zhao_2015,Roberto_2022,Zhao_Meng_2018}).

In this paper, we address the control problem for the Kawahara equation \eqref{equkaw} on the circle \(\mathbb{T} := \mathbb{R}/2\pi \mathbb{Z}\). Without loss of generality, we assume the parameters \(\alpha = -1\), \(\gamma = \frac{1}{3}\), and \(\beta \in \{-1, 0, 1\}\). By introducing the transformation \(w = u - a\), where  

\[
a = [u_0(x)] = \frac{1}{2\pi} \int_{\mathbb{T}} u_0(x) \, dx,
\]

we observe that if \(u\) solves \eqref{equkaw}, then \([u(t, \cdot)] \equiv a\) for all \(t \geq 0\), and \(w\) satisfies the following equation:  

\begin{align}\label{shift_equn}
	\begin{cases}
		w_t - w_{xxxxx} + \beta w_{xxx} + a w_x + w w_x = 0, & \text{for } t > 0, \, x \in \mathbb{T}, \\
		w(0, x) = w_0(x), &
	\end{cases}
\end{align}

where \([w(t, \cdot)] \equiv 0\) for all \(t \geq 0\). This motivates the study of the equation in the mean-zero Sobolev space, defined as  

\[
H^s_0(\mathbb{T}) := \left\{ u \in H^s(\mathbb{T}) : [u] = \frac{1}{2\pi} \int_{\mathbb{T}} u(x) \, dx = 0 \right\},
\]

where \(H^s(\mathbb{T})\) denotes the Sobolev space of order \(s\) equipped with the standard norm \(\| \cdot \|_s\). For simplicity, the \(L^2(\mathbb{T})\) norm (\(H^0(\mathbb{T})\)) is denoted by \(\| \cdot \|\).

The Cauchy problem for \eqref{shift_equn} has been extensively studied for its well-posedness in \(H^s_0(\mathbb{T})\), following the foundational work on the KdV equation (see \cite{Cui_Tao_2005,Cui_Tao_2006,Chen_Wengu_Zihua_2011} and references therein). The best known result to date \cite{Kato_2012} establishes that the problem is locally well-posed in \(H^s_0(\mathbb{T})\) for any \(s \geq -\frac{3}{2}\) and globally well-posed in \(H^s_0(\mathbb{T})\) for \(s \geq -1\).

To simplify the calculations, we set \(\beta = 1\) and examine the following control problem:

\begin{align}\label{ctrl_equn}
	\begin{cases}
		u_t - u_{xxxxx} + u_{xxx} + u u_x = \eta(t, x), & \text{in } (0, T) \times \mathbb{T}, \\
		u(0, x) = u_0(x), & \text{in } \mathbb{T},
	\end{cases}
\end{align}

where \(T > 0\), \(u_0\) is the initial condition, and \(\eta\) represents the control input. We aim to analyze the approximate controllability of \eqref{ctrl_equn}.

In \cite{Araruna_Capistrano_Doronin_2012}, the authors recently investigated the stabilization problem and proposed a critical set phenomenon for the Kawahara equations, similar to that observed in the KdV equation \cite{Rosier_1997} and the Boussinesq KdV–KdV system \cite{Capistrano_Roberto_Pazoto_2019}, among others. Furthermore, to the best of our knowledge, the control problem \eqref{ctrl_equn} was first addressed in \cite{Zhang_Zhao_2012,Zhang_Zhao_2015}, where the authors studied the Kawahara equation on a periodic domain $\mathbb{T}$ with a distributed control. In \cite{Zhang_Zhao_2012}, the authors successfully established local controllability results for this equation, specifically proving the following theorem:
\begin{customthm}{A}[\cite{Zhang_Zhao_2012}]
	Let \( s \geq 0 \) and \( T > 0 \) be given. There exists a \(\delta > 0\) such that, for any \( u_0, u_1 \in H^s(\mathbb{T}) \) satisfying \( [u_0] = [u_1] \) and 
	\[
	\| u_0 \|_s \leq \delta, \quad \| u_1 \|_s \leq \delta,
	\]
	there exists a control input \(\eta \in L^2(0, T; H^s(\mathbb{T}))\) with support in an open set $\omega\subset\mathbb{T}$, such that the system \eqref{ctrl_equn} has a solution \( u \in C([0, T]; H^s(\mathbb{T})) \) satisfying
	\[
	u(T, \cdot) = u_1(\cdot).
	\]
\end{customthm}
This result was further extended to semi-global controllability for sufficiently large \( T > 0 \) in \cite{Zhang_Zhao_2015}, as stated below:  

\begin{customthm}{B}[\cite{Zhang_Zhao_2015}]  
	Let \( R > 0 \) be given. There exists a \( T > 0 \) such that, for any \( u_0, u_1 \in H^s(\mathbb{T}) \) (\( s \geq 0 \)) satisfying \( [u_0] = [u_1] \) and  
	\[
	\| u_0 \| \leq R, \quad \| u_1 \| \leq R,
	\]  
	there exists a control input \(\eta \in L^2(0, T; H^s(\mathbb{T}))\) with support in an open set $\omega\subset\mathbb{T}$, such that the system \eqref{ctrl_equn} admits a solution \( u \in C([0, T]; H^s(\mathbb{T})) \) satisfying  
	\[
	u(T, \cdot) = u_1(\cdot).
	\]  
\end{customthm}  

Despite these advances, the question of global controllability (without constraints on the initial data, terminal data, and the time of controllability) for \eqref{ctrl_equn} remains unsolved. This study aims to establish global controllability at any time \( T > 0 \) using finite-dimensional control. However, this approach sacrifices exact controllability, achieving only approximate controllability. Moreover, the applied control is localized in Fourier modes, unlike in the previously mentioned theorem, where the control is localized in space.

We begin with initial data in \( H^s_0(\mathbb{T}) \) and a control that comes from a finite-dimensional space defined as  
\begin{align}\label{span_set}
	\mathcal{H}(I) := \text{span}_{\mathbb{R}}\{ \sin (kx), \cos (kx) : k \in I  \},   
\end{align} 
where \( I \subset \mathbb{Z} \) is a finite set. These types of controls are not only of theoretical interest but also find broad applications in physics and engineering.
\begin{definition}[Approximate controllability]\label{approx_ctrl_defn}  
	The equation \eqref{ctrl_equn} is said to be approximately controllable in \( H^s_0(\mathbb{T}) \) by values in \( \mathcal{H}(I) \) if, for any \(T>0,  \varepsilon > 0 \) and any \( u_0, u_1 \in H^s_0(\mathbb{T}) \), there exists a piecewise constant control \( \eta \) with values in \( \mathcal{H}(I) \) and a solution \( u \) of \eqref{ctrl_equn} such that  
	\[
	\| u(T) - u_1 \|_{s} \leq \varepsilon.
	\]  
\end{definition}  
Before presenting our main result, we first introduce the following definitions: Let \( I \subset \mathbb{Z} \). We say that \( I \) is a generator if every element of \( \mathbb{Z} \) can be expressed as a linear combination of elements in \( I \) with integer coefficients. Additionally, \( I \) is called symmetric if \( I = -I \).

Now, we state our main result in the following theorem:
\begin{theorem}\label{main_thm}
	Let \( T > 0 \), \( s \in \mathbb{N} \), and \( I \subset \mathbb{Z}^* \) be a finite symmetric set. The equation \eqref{ctrl_equn} is approximately controllable in \( H^s_0(\mathbb{T}) \) by a piecewise constant control taking values in \( \mathcal{H}(I) \) if and only if \( I \) is a generator of \( \mathbb{Z} \). 
\end{theorem}

\begin{remark} 
The smallest symmetric generating subset of \( \mathbb{Z}^* \) is \( \{\pm 1\} \). Consequently, global approximate controllability at any time can be achieved using control values in a two-dimensional space. The control is expressed as  
	\[  
	\eta(t, x) = \alpha_1(t) \sin(x) + \alpha_2(t) \cos(x),  
	\]  
	where the functions \( \alpha_i \in L^2(\mathbb{R}^+, \mathbb{R}) \) for \( i = 1, 2 \) act as controls.  
\end{remark}
To achieve the desired result, we adopt the Agrachev-Sarychev approach, initially introduced by Agrachev and Sarychev \cite{Agrachev_Sarychev_2005,Agrachev_Sarychev_2006} to study the approximate controllability of the Navier-Stokes and 2D Euler systems under finite-dimensional forcing. 
This approach has been successfully extended to various equations. It has been applied to the 3D Navier-Stokes equations \cite{Shirikyan_06,Shirikyan_07}, as well as to both the compressible and incompressible Euler equations \cite{Nersisyan_2010,Nersisyan_2011}. Its application has also been explored for the viscous Burgers equation \cite{Shirikyan_14,Shirikyan_18} and the Schrödinger equation \cite{Sarychev_12}. Additionally, it has been employed to demonstrate small-time approximate controllability for the quantum density and momentum of the 1D semiclassical cubic Schrödinger equation \cite{Coron_Xiang_Zhang_2023}. More recent applications on nonlinear parabolic PDEs \cite{Narsesyan_21}, which builds on the contributions of \cite{Glatt-Holtz_Herog_Mattingly_2018}. This technique has also been utilized to establish global approximate controllability using finite-dimensional control for various equations, including the Camassa-Holm Equation \cite{Shirshendu_Debanjit_2024}, the Benjamin-Bona-Mahony equation \cite{Jellouli_23}, and the Kuramoto-Sivashinsky equation \cite{Peng_Gao_2022}.
To the best of our knowledge, the Agrachev-Sarychev approach has not been previously applied to the Kawahara equation. Unlike parabolic equations, the Kawahara equation presents challenges in establishing smoothing properties in classical Sobolev spaces \( H^s(\mathbb{T}) \). Consequently, many of the crucial estimates required for the Agrachev-Sarychev approach cannot be obtained using techniques developed for parabolic equations. To address this difficulty, we introduce a functional space constructed by Bourgain \cite{Bourgain_93}, now known as the Bourgain space associated with the Kawahara equation. By leveraging the favorable properties of the Bourgain space and the conserved integral quantities of the Kawahara equation, we are able to adapt the Agrachev-Sarychev approach to this setting.

The proof of Theorem \ref{main_thm} is divided into two parts: the sufficient part and the necessary part. The sufficient part begins by considering a symmetric subset $I$ of $\mathbb{Z^*}$ which generates $\mathbb{Z}$. We then define a sequence of subsets $\{I_n\}$ such that $I_n \subsetneq I_{n + 1},$ and the union $\cup_{n \in \mathbb{N}} I_n = \mathbb{Z}.$ Consequently, the union of the corresponding spaces $\mathcal{H}(I_n)$ contains all Fourier modes. This implies that $\cup_{n \in \mathbb{N}} \mathcal{H}(I_n)$ is dense in $H^s_0(\mathbb{T}).$ To establish approximate controllability in $H^s_0(\mathbb{T}),$ it suffices to consider a target state in an affine finite-dimensional space shifted by the initial data $u_0,$ namely $u_0 + \mathcal{H}(I_n)$ for some $n \in \mathbb{N}.$ A crucial observation is that the nonlinear term $u \partial_{x}u$ behaves linearly with respect to the frequency of Fourier modes. Exploiting this fact, along with certain asymptotic properties, we can apply a large $\mathcal{H}(I_0)$-valued control over a short time interval to steer the trajectory of \eqref{ctrl_equn} arbitrarily close to any target in $u_0 + \mathcal{H}(I_0).$ By iterating this process, we can extend the reach to any target in $u_0 + \mathcal{H}(I_n)$ for larger $n,$ because the structure of the nonlinearity ensures that each step allows us to access a larger affine space. Ultimately, this process establishes approximate controllability in small time. By applying the continuity and stability of the solution, we ensure that the trajectory stays close to the target over a longer time period, leading to approximate controllability within the specified time. For the necessary part, we assume that $I \subset \mathbb{Z^*}$  is a symmetric subset that does not generate $\mathbb{Z}.$ In this case, the trajectory of \eqref{ctrl_equn} under $\mathcal{H}(I)$-valued control cannot reach a dense subset of $H^s_0(\mathbb{T}).$ Hence, for approximate controllability, $I$ must generate $\mathbb{Z}.$

The article is organized as follows: Section \ref{sec_exit} presents the well-posedness result for the nonlinear problem, along with propositions regarding the continuity of the solution with respect to the initial data and the asymptotic behavior, which are crucial for proving the main result. In Section \ref{sec_algebra}, we establish certain algebraic properties of \( I_n \), \( \mathcal{H}(I_n) \), and the solution map. Section \ref{sec_mainthm} is devoted to proving the main result, Theorem \ref{main_thm}. Finally, in Section \ref{sec_prop}, we review properties of Bourgain spaces associated with the Kawahara equation and use these to establish the results stated in Section \ref{sec_exit}.


\section{Well-posedness and Asymptotic Property}\label{sec_exit} 

In this section, we discussed results about the existence of solutions and the stability of the equation. These results are essential for establishing our main result, Theorem \ref{main_thm}. Detailed proofs of these results will be provided in Section \ref{sec_prop}. For now, let's look at the following generalization of the system \eqref{ctrl_equn}:
\begin{align}\label{xtd_equn}
	\begin{cases}
		u_t - (u + \zeta)_{xxxxx} + (u + \zeta)_{xxx} + (u + \zeta)(u + \zeta)_x = \varphi, &\text{ in } (0,T) \times \mathbb{T}, \\
		u(0, \cdot) = u_0(\cdot), &\text{ in } \mathbb{T},
	\end{cases}
\end{align}
where $\zeta = \zeta(x)$ and $\varphi = \varphi(t, x)$ are given functions. 


\begin{proposition}[Well-posedness]\label{prp_existance}
	Let \(s \in \mathbb{N}\) and \(T > 0\). For \(u_0 \in H^s(\mathbb{T}),\ \zeta \in H_0^{s + 7}(\mathbb{T}),\) and \(\varphi \in L^2(0, T; H_0^s(\mathbb{T})),\) system \eqref{xtd_equn} admits a unique solution in \(C([0, T]; H^s(\mathbb{T})).\)
\end{proposition}
Define $\mathcal{R}$ as the mapping that takes the triple $(u_0, \zeta, \varphi)$ to the solution of \eqref{xtd_equn}. The restriction of \(\mathcal{R}(u_0, \zeta, \varphi)\) at time \(t\) is denoted by \(\mathcal{R}_t(u_0, \zeta, \varphi)\).
\begin{remark}\label{rmk_lip}
	From the uniqueness of the solution to \eqref{xtd_equn}, we obtain the following equality for all \(t \in [0, T]\):
	\begin{equation}\label{uniqueness_property}
		\mathcal{R}_t(u_0, \zeta, \varphi) = \mathcal{R}_t(u_0 + \zeta, 0, \varphi) - \zeta.
	\end{equation}    
\end{remark}
\begin{proposition}[Stability]\label{prp_continty}
	Let \(s \in \mathbb{N},\ T > 0,\) and \(\varphi \in L^2(0, T; H^s_0(\mathbb{T}))\). For any \(u_0, v_0 \in H^s(\mathbb{T})\) such that $\left[ u_0\right]=\left[ v_0\right], $ there exists a constant \(C > 0\) such that
	\begin{align}\label{lips_property}
		\sup_{t \in [0, T]} \|\mathcal{R}_t(u_0, 0, \varphi) - \mathcal{R}_t(v_0, 0, \varphi)\|_s \leq C \|u_0 - v_0\|_s.
	\end{align}
\end{proposition}

 To investigate approximate controllability, we study the following perturbed system. For \(\delta > 0\) and two smooth functions \(\zeta(x)\) and \(\eta(x)\), we consider the equation:
\begin{align}\label{xtd_dlta_equn}
	\begin{cases}
		u_t - (u + \delta^{-\frac{1}{2}} \zeta)_{xxxxx} + (u + \delta^{-\frac{1}{2}} \zeta)_{xxx} + (u + \delta^{-\frac{1}{2}} \zeta)(u + \delta^{-\frac{1}{2}} \zeta)_x = \delta^{-1}\eta, &\text{ for } t > 0,\ x \in \mathbb{T}, \\
		u(0, x) = u_0(x).
	\end{cases}
\end{align}

\begin{proposition}[Asymptotic property]\label{prp_limit}
	Let \(s \in \mathbb{N}\). For any \(u_0 \in H^{s + 7}(\mathbb{T}),\ \zeta \in H_0^{s + 8}(\mathbb{T}),\) and \(\eta \in H_0^{s + 7}(\mathbb{T}),\) there exists \(\delta_0 \in (0, 1)\) such that for any \(\delta \in (0, \delta_0),\) the solution of \eqref{xtd_dlta_equn} satisfies the following limit:
	\[
	\mathcal{R}_\delta(u_0, \delta^{-\frac{1}{2}} \zeta, \delta^{-1} \eta) \to u_0 + \eta - \zeta \zeta_x, \quad \text{in } H^s(\mathbb{T}), \text{ as } \delta \to 0^+.
	\]
\end{proposition}

The asymptotic property described above demonstrates that the controlled trajectory starting from \(u_0\) can approach the affine set \(u_0 + \mathcal{V}\) in a short time, where the set \(\mathcal{V}\) is determined by the nonlinearity of the equation. This insight motivates us to introduce the following subset:  

Recalling the definition of \(\mathcal{H}(I)\) from \eqref{span_set}, we now define the set 
\begin{align}\label{C(I)_defn}
	\mathcal{C}(I) := \Bigg\{\eta - \sum_{i=1}^d \zeta_i \partial_x \zeta_i \; \bigg| \; \eta, \zeta_i \in \mathcal{H}(I), \ \forall d \geq 1\Bigg\}.
\end{align}  
It is straightforward to observe that  
\[
\mathcal{H}(I) \subset \mathcal{C}(I), \quad \text{for all } I \subset \mathbb{Z}.
\]  

Moreover, there exists a beautiful and crucial relationship between the finite-dimensional subspace \(\mathcal{H}(I_{n + 1})\) and the set \(\mathcal{C}(I_n)\), which will be rigorously studied in the next section.

\section{Some Algebraic Results.}\label{sec_algebra}
We begin this section by presenting a fundamental result from algebra. Let \( I \) be a finite subset of \( \mathbb{Z} \). Define:  
\[
\mathcal{L}_{\mathbb{Z}}(I) := \left\{ \sum_{i=1}^{m} \lambda_i a_i \;\middle|\; \lambda_i \in \mathbb{Z}, \ a_i \in I, \ m \geq 1 \right\}.
\]
The set \( I \) is called a generator of \( \mathbb{Z} \) if \( \mathcal{L}_{\mathbb{Z}}(I) = \mathbb{Z} \). A subset \( I \) is a generator of \( \mathbb{Z} \) if and only if \(\gcd(a_1, a_2, \ldots, a_n) = 1\), where \( a_1, a_2, \ldots, a_n \in I \) are nonzero. To avoid ambiguity, we allow divisors to include negative integers as well.

Now, suppose a symmetric generator \( I \subset \mathbb{Z}^* \) is given. Define a sequence of subsets by induction for all \( n \in \mathbb{N} \):  
\[
\begin{aligned}
	& I_0 := I, \quad \text{and} \\
	& I_{n+1} := \{ i + j \mid i, j \in I_n \}.
\end{aligned}
\]
The following lemma will play a crucial role in demonstrating that we can reach all the Fourier modes to achieve the approximate controllability.

\begin{lemma}\label{lema_regarding_I_n}
If $I_0 \subset \mathbb{Z^*}$ be a symmetric generator of $\mathbb{Z}$ then, the followings are satisfied :
\begin{enumerate}[(i)]
\item For every $\alpha \in \mathbb{Z},$ there exists $n \in \mathbb{N}$ such that $\alpha \in I_n,$ i.e. , 
\begin{equation}\label{union_Z}
\bigcup\limits_{n \in \mathbb{N}} I_n = \mathbb{Z}.
\end{equation}
\item For all $I_n,$ we define $\mathcal{H}(I_n)$ and $\mathcal{C}(I_n)$ as \eqref{span_set} and \eqref{C(I)_defn} respectively then 
\begin{enumerate}[(a)]
\item for all $n \ge 1,$ we have the inclusion property 
\begin{align}\label{H(I_n+1)_subset_C(I_n)}
\mathcal{H}(I_{n + 1}) \subset C(I_n),
\end{align}	
\item for $n = 0,$ we only have 
\begin{align}\label{n+0_inclusion}
\text{ span}_{\mathbb{R}}\Big\{\sin(kx), \cos(kx); k \text{ is nonzero and belongs to } I_1\Big\} \subset C(I_0).
\end{align}
\end{enumerate}
\end{enumerate}
\end{lemma}

An immediate conclusion of the Lemma \ref{lema_regarding_I_n} is that if $I_0$ is a generator of $\mathbb{Z}$, the space  $$\mathcal{H}_{\infty}^{I} := \bigcup\limits_{n\in\mathbb{N}} \mathcal{H}(I_{n})$$  is dense in $H^s_0(\mathbb{T}),$ then for all $u_0, u_1 \in H^s_0(\mathbb{T}),$ we can find  large enough $N(\varepsilon) \in \mathbb{N}$  and an element $$u_{N(\varepsilon)} \in \mathcal{H}(I_{N(\varepsilon) + 1}) \subset C(I_{N(\varepsilon)})$$ such that $\|(u_1 - u_0) - u_{N(\varepsilon)}\|_s \le \varepsilon,$ then by Proposition \ref{prp_limit}, for the particular choice of $(\zeta, \eta) \equiv (0, u_{N(\varepsilon)})$ we obtain $$\mathcal{R}_{\delta}(u_0, 0, \delta^{-1} \eta) \to u_0 + u_{N(\varepsilon)} \ \text{ in } H^s, \ \text{ as } \delta \to 0^+.$$ Furthermore, by Proposition \ref{prp_continty}, we can steer the initial data $u_0$ close to $u_1$ in a small time by a $\mathcal{H}(I_{N(\varepsilon) + 1})$-valued control. This motivates us to investigate approximate controllability using controls that take values in a finite-dimensional subspace. In the next section, we demonstrate approximate controllability with controls confined to a finite-dimensional subspace, and moreover, we establish that the dimension of this subspace is optimal. Before move to that in details let us prove the Lemma \ref{lema_regarding_I_n}.

\begin{proof}[Proof of Lemma \ref{lema_regarding_I_n}]
(i). For \( n \in \mathbb{N}^* \), define:  
\[
\Gamma_n = \left\{ \sum_{j=1}^l z_j \tau_j \;\middle|\; z_j \in \mathbb{Z}, \ \tau_j \in I_0, \ \sum_{j=1}^l |z_j| \leq n, \ l \geq 1 \right\}.
\]  
Since \( I_0 \) is a generator of \( \mathbb{Z} \), it is evident that:  
\[
\bigcup_{n \in \mathbb{N}^*} \Gamma_n = \mathbb{Z}.
\]  
We now prove by induction that for all \( n \in \mathbb{N}^* \), \( \Gamma_n \subset I_{n-1} \).  

For \( n = 1 \), since \( I_0 \) is symmetric, we have \( \Gamma_1 = \{ \pm \tau \mid \tau \in I_0 \} = I_0 \). Thus, \( \Gamma_1 \subset I_0 \).  

Assume that \( \Gamma_n \subset I_{n-1} \) for some \( n > 1 \) with \( n \in \mathbb{N}^* \). Let \( p = \sum\limits_{j=1}^l z_j \tau_j \in \Gamma_{n+1} \). If \( \sum\limits_{j=1}^l |z_j| \leq n \), then \( p \in \Gamma_n \subset I_{n-1} \). Since \( 0 \in I_{n-1} \) for \( n \geq 2 \), it follows that \( p \in I_n \).

Now consider the case where \( \sum\limits_{j=1}^l |z_j| = n + 1 \). Without loss of generality, assume \( z_j \neq 0 \) for all \( j = 1, \ldots, l \). Define:  
\[
p' = (z_1 \pm 1)\tau_1 + \sum_{j=2}^l z_j \tau_j,
\]  
where \( |z_1 \pm 1| = |z_1| - 1 \). Then:  
\[
|z_1 \pm 1| + \sum_{j=2}^l |z_j| = n,
\]  
which implies \( p' \in \Gamma_n \). By the induction hypothesis, \( p' \in I_{n-1} \). Since \( \tau_1 \in I_0=\Gamma_1\subset\Gamma_n \subset I_{n-1} \), it follows from the definition of \( I_n \) that:  
\[
\tau_1 \mp \left( (z_1 \pm 1)\tau_1 + \sum_{j=2}^l z_j \tau_j \right) = \tau_1 \mp p' = p\in I_n.
\]  

Hence \( \Gamma_{n+1} \subset I_n \). By induction, \( \Gamma_n \subset I_{n-1} \) holds for all \( n \in \mathbb{N}^* \). Consequently:  
\[
\mathbb{Z}^* \subset \bigcup_{n \in \mathbb{N}} I_n.
\]  
This establishes the desired equality.

\noindent(ii). Let us first establish part (a). Take, $n \in \mathbb{N^*}$, as $I_0 \subset  \mathbb{Z^*}$ be symmetric, then $0 \in I_n.$ Now, using trigonometric identities, we obtain that for $\alpha, \beta \in \mathbb{Z},$
\begin{align}
&2 \mathcal{B}\Big(\sin (\alpha x) + \sin (\beta x)\Big) = \alpha \sin (2 \alpha x) + \beta \sin (2 \beta x) + (\alpha + \beta) \sin ((\alpha + \beta)x) - (\alpha - \beta) \sin ((\alpha - \beta) x), \label{sin_alpha+sin_beta}\\
&2 \mathcal{B}\Big(\cos (\alpha x) + \cos (\beta x)\Big) =  - \alpha \sin (2 \alpha x) - \beta \sin (2 \beta x) - (\alpha + \beta) \sin ((\alpha + \beta)x) - (\alpha - \beta) \sin ((\alpha - \beta) x), \label{cos_alpha+cos_beta}\\
&2 \mathcal{B}\Big(\sin (\alpha x) \pm \cos (\beta x)\Big) = \alpha \sin (2 \alpha x) - \beta \sin(2\beta x) \pm (\alpha + \beta) \cos ((\alpha + \beta)x) \pm (\alpha - \beta) \cos ((\alpha - \beta)x), \label{sin_alpha_pm_cos_beta}
\end{align}
where $\mathcal{B}(u) = u \partial_x u,$ the nonlinearity of the equation \eqref{ctrl_equn}. 

 Our aim to prove for $$ \text{ span}_{\mathbb{R}} \Big\{\sin (kx), \cos (kx) \ ; k \in I_{n + 1} \Big\} \subset \left \{ \eta - \sum_{i = 1}^{d} \zeta_i \partial_{x} \zeta_i \ ; \eta, \zeta_i \in \mathcal{H}(I_n) \right\}.$$ 

\noindent If $k = 2j,$ for some $j \in I_n,$ then
\begin{enumerate}
\item choosing $(\alpha, \beta) = (j, 0)$ in \eqref{sin_alpha+sin_beta} and $(\alpha, \beta) = (0, j)$ in \eqref{sin_alpha_pm_cos_beta}, we obtain $$\text{span}_{\mathbb{R}} \{ \sin (kx)\} \subset \mathcal{C}(I_n),$$
\item  taking $(\alpha, \beta) = (j, j)$ in \eqref{sin_alpha_pm_cos_beta}, we have $$\text{span}_{\mathbb{R}} \{ \cos (kx)\} \subset \mathcal{C}(I_n).$$
\end{enumerate}
If $k = (j \pm p),$ for some $j, p \in I_n$ and $(j \pm p) \ne 0.$ 
 Choosing $(\alpha, \beta) = (j, \pm p)$ in \eqref{sin_alpha+sin_beta}, \eqref{cos_alpha+cos_beta}  we see $$\text{span}_{\mathbb{R}} \{ \sin ((j \pm p)x)\} \subset \mathcal{C}(I_n),$$ where we have used the symmetricty of $I_n.$ By similar argument  $$\text{span}_{\mathbb{R}} \{ \cos ((j \pm p)x)\} \subset \mathcal{C}(I_n),$$ 
If $k = (j \pm p),$ for some $j, p \in I_n$ and $(j \pm p) = 0.$ Then the vector space, $\text{span}_{\mathbb{R}} \{ 1 \} \subset \mathcal{H}(I_{n + 1}).$ As $0 \in I_n$ then  $\text{span}_{\mathbb{R}} \{ 1 \} \subset \mathcal{H}(I_n) \subset \mathcal{C}(I_n),$ where we have essentially used the fact that $n \ge 1.$ Hence the inclusion is true for any $n \in \mathbb{N^*}.$

\noindent To prove part (b), perform a similar analysis on $I_0$ and $I_1$ as in part (a), with the exception of the case $(j \pm p) = 0.$ In this scenario, inclusion is not possible because when $(j + p) = 0,\ 0 \not \in  I_0$ (as  $I_0$ is a symmetric subset of  $\mathbb{Z^*}$), whereas $0 \in I_1.$ More over, for any $d \ge 1$ and $s_j, c_j, s^{i}_j, c^{i}_j \in \mathbb{R}$ one can see, $$1 \ne \displaystyle \sum_{j \in I_0, \atop j > 0}\Big(s_j \sin (jx) + c_j \cos (jx) \Big) + \sum_{i = 1}^{d} \mathcal{B}\left(\displaystyle \sum_{j \in I_0, \atop j > 0}\Big(s^{i}_j \sin (jx) + c^{i}_j \cos (jx) \Big) \right).$$ Consequently, 
$\text{span}_{\mathbb{R}} \{ 1 \} \subset \mathcal{H}(I_1),$ but $\text{span}_{\mathbb{R}} \{ 1 \} \not \subset \mathcal{C}(I_0).$ 
\end{proof}

We finish this section by giving an algebraic property of the flow map $\mathcal{R}_{t}$
\begin{lemma}\label{alg_r(t)}
Let $\mathcal{R}_t(u_0, 0, \eta)$ be the solution of \eqref{xtd_equn}, Where $\eta$ is given by $$\eta(s) = \mathbbm{1}_{[0, t_1)} \eta_1(s) + \mathbbm{1}_{[t_1, t_1 + t_2)} \eta_2(s) + \mathbbm{1}_{[t_1 + t_2, t_1 + t_2 + t_3)} \eta_3(s) $$ For all $t_1, t_2, t_3 \ge 0,$ we have the equality $$ \mathcal{R}_{t_1 +t_2 +t_3}(u_0,0, \eta) = \mathcal{R}_{t_3}(\mathcal{R}_{t_2}(\mathcal{R}_{t_1}(u_0, 0, \eta_1(\cdot)), 0, \eta_2(\cdot - t_1)),0, \eta_3(\cdot - t_2-t_1)).$$
\end{lemma}

\begin{proof} Let $0<s\le t$ and $\widehat{\mathcal{R}}(t, s, v, \eta)$  be the solution of \eqref{xtd_equn} at the instant $t,$ when $\zeta = 0$ and with initial data $\widehat{\mathcal{R}}(s, s, v, \eta) = v.$ That means $\mathcal{R}_{t}(u_0, 0, \eta ) = \widehat{\mathcal{R}}(t, 0, u_0, \eta).$ From the uniqueness of the solutions, we see that for all $\sigma \in [s, t]$,
\begin{align}\label{uni_R}
\widehat{\mathcal{R}}(t, \sigma, \widehat{\mathcal{R}}(\sigma, s, u_0, \eta), \eta ) = \widehat{\mathcal{R}}(t, s, u_0, \eta).	
\end{align}
Using \eqref{uni_R} we can write
\begin{align*}
&\widehat{\mathcal{R}}(t_1 + t_2 + t_3, 0, u_0, \eta)\\
& = \widehat{\mathcal{R}}(t_1 + t_2 + t_3, t_1 + t_2,\widehat{\mathcal{R}}(t_1 + t_2, t_1, \widehat{\mathcal{R}}(t_1, 0,  u_0, \eta_1(\cdot)), \eta_2(\cdot - t_1)), \eta_3(\cdot - t_2 -t_1))\\
& = \widehat{\mathcal{R}}(t_3, 0,\widehat{\mathcal{R}}(t_1 + t_2, t_1, \widehat{\mathcal{R}}(t_1, 0,  u_0, \eta_1(\cdot)), \eta_2(\cdot - t_1)), \eta_3(\cdot - t_2 -t_1)) \\
& = \mathcal{R}_{t_3}(\widehat{\mathcal{R}}(t_1 + t_2, t_1, \widehat{\mathcal{R}}(t_1, 0,  u_0, \eta_1(\cdot)), \eta_2(\cdot - t_1)), 0,  \eta_3(\cdot - t_2 -t_1)) \\
& = \mathcal{R}_{t_3}(\mathcal{R}_{t_2}(\mathcal{R}_{t_1}(u_0, 0, \eta_1(\cdot)), 0,  \eta_2(\cdot - t_1)), 0,  \eta_3(\cdot - t_2 -t_1)).
\end{align*}
	This completes the proof.
\end{proof}

\section{Approximate Controllability}\label{sec_mainthm}
We begin this section by introducing the concept of small-time approximate controllability for the control system \eqref{ctrl_equn}. The formal definition is presented below.
\begin{definition}\label{approxi_cntrl_small time}
Let $H$ be a Hilbert space and $\mathcal{G}$ be a vector space. We say that the equation \eqref{ctrl_equn} is small time approximately controllable in $H$ by the control values in $\mathcal{G}$ if for any $ \varepsilon > 0, \ \sigma > 0$ and any $u_{0}, u_1 \in H$, there exists a $\delta \in (0, \sigma)$ and a control $\eta \in L^{2}(0, \delta ; \mathcal{G})$ such that the solution $u$ of \eqref{ctrl_equn} satisfies 
$$\| \mathcal{R}_{\delta}(u_0, 0, \eta) - u_1\|_{H} \le \varepsilon.$$
\end{definition}
The main theorem follows from the proposition stated below, which is motivated by \cite{Jellouli_23}. The proof of the proposition is provided at the end of this section.
\begin{proposition}\label{prp_small_time_ctrl}
Let $s \in \mathbb{N}$ and $I $ be the finite subset of $\mathbb{Z^*}$ satisfying the following conditions :
\begin{itemize}
	\item[(i)] $I$ is symmetric,
	\item[(ii)] $I$ is a generator of $\mathbb{Z},$
\end{itemize}
then equation \eqref{ctrl_equn} is small time approximate controllable in $H^s_0(\mathbb{T})$ by a control values in $\mathcal{H}(I).$
\end{proposition}

\begin{proof}[\bf Proof of Theorem \ref{main_thm}]
	We continue the proof in following parts:
	
\noindent \textbf{Sufficient condition for controllability :} Here, we prove that if the given symmetric subset $I \subset \mathbb{Z^*}$ generates $\mathbb{Z},$ then equation \eqref{ctrl_equn} is approximate controllable by $\mathcal{H}(I)$ valued control in the sence of Definition \ref{approx_ctrl_defn}. Specifically, we aim to show that for given $\varepsilon, T > 0$ and $u_0, u_1 \in H^s_0(\mathbb{T}),$ there exists a control $\eta \in L^2(0, T; \mathcal{H}(I))$ such that solution of equation \eqref{ctrl_equn} satisfy 
\begin{align}\label{aim_any_time_ctrl}
\| \mathcal{R}_T(u_0, 0, \eta) - u_1 \|_s \le \varepsilon.
\end{align}
By the Proposition \ref{prp_small_time_ctrl} we have small-time approximate controllability for equation \eqref{ctrl_equn} using $\mathcal{H}(I)$-valued control. Then the key idea here is to design a $\mathcal{H}(I)$-valued control such that the corresponding trajectory remains arbitrarily close to $u_1$ for a given time duration.

First utilizing the time continuity of the solution and the continuous dependence on the initial data (Proposition \ref{prp_continty}),  we deduce that, there exist $r \in (0, \varepsilon)$ and $\tau > 0$ such that 
	\begin{align}\label{time_continty}
		\| \mathcal{R}_t(v, 0, 0) - u_1 \|_s \le \varepsilon, \ \text{ for all } t \in [0, \tau], \text{ and } v \in B_{H^s_0(\mathbb{T})}(u_1, r),
	\end{align}
	this says, the control free trajectory starting from $u_1$ remains close to $u_1$ for a certain time interval. 
	
	Now applying the Proposition \ref{prp_small_time_ctrl}, we get a small enough $\delta_1 \in (0, T)$ and a control $\hat{\eta}_1 \in L^2(0, \delta_1; \mathcal{H}(I))$ such that,
	\begin{align}\label{conclusion_small_time_ctrl}
		\| \mathcal{R}_{\delta_1}(u_0, 0, \hat{\eta}_1) - u_1\|_s \le r.
	\end{align}
	
			\begin{figure}[h!]
		\centering
		\begin{tikzpicture}[x=0.7 cm, y=0.7cm]
			\begin{pgfonlayer}{nodelayer}
				\node [style=none] (0) at (-1.5, 0) {};
				\node [style=none] (1) at (1.5, 0) {};
				\node [style=none] (2) at (0, 1.5) {};
				\node [style=none] (3) at (0, -1.5) {};
				\node [style=none] (4) at (0, 3.5) {};
				\node [style=none] (5) at (-3.5, 0) {};
				\node [style=none] (6) at (0, -3.5) {};
				\node [style=none] (7) at (3.5, 0) {};
				\node [style=Dotty] (8) at (0, 0) {};
				\node [style=Dotty] (9) at (-6.75, -3.75) {};
				\node [style=none] (10) at (-0.75, 0.25) {};
				\node [style=none] (11) at (-1, -2.5) {};
				\node [style=none] (12) at (1, -0.25) {};
				\node [style=none] (13) at (2.75, 1.5) {};
				\node [style=none] (14) at (0.5, 1) {};
				\node [style=none] (15) at (-1.75, 3) {};
				\node [style=none] (16) at (0.25, 0.85) {$r$};
				\node [style=none] (17) at (-1, 2) {$\varepsilon$};
				\node [style=none] (18) at (-4, -2.75) {$\delta_1$};
				\node [style=none] (19) at (-4.5, -1.90) {$\hat{\eta}_1$};
				\node [style=none] (20) at (-1.5, -1.25) {$\tau$};
				\node [style=none] (21) at (0.25, -0.75) {$\hat{\eta}_2$};
				\node [style=none] (22) at (-0.3, -2.00) {$\delta_2$};
				\node [style=none] (23) at (-6.75, -4.5) {$u_0$};
				\node [style=none] (24) at (0.5, 0) {$u_1$};
				\node [style=none] (25) at (-8.5, 3) {};
				\node [style=none] (26) at (-9.5, 3) {};
				\node [style=none] (27) at (-9.5, 2) {};
				\node [style=none] (28) at (-8.5, 2) {};
				\node [style=none] (29) at (-5.4, 3) {Controlled Trajectory};
				\node [style=none] (30) at (-6, 2) {Free Trajectory};
			\end{pgfonlayer}
			\begin{pgfonlayer}{edgelayer}
				\draw [bend left=45] (0.center) to (2.center);
				\draw [bend left=45] (2.center) to (1.center);
				\draw [bend left=45] (1.center) to (3.center);
				\draw [bend left=45] (3.center) to (0.center);
				\draw [bend left=45] (4.center) to (7.center);
				\draw [bend left=45] (7.center) to (6.center);
				\draw [bend left=45] (6.center) to (5.center);
				\draw [bend left=45] (5.center) to (4.center);
				\draw [in=165, out=0, looseness=1.25] (9) to (10.center);
				\draw [in=135, out=-30, looseness=1.25,dashed] (10.center) to (11.center);
				\draw [in=75, out=-105, looseness=1.25] (12.center) to (11.center);
				\draw [in=-90, out=90, looseness=1.25,dashed] (12.center) to (13.center);
				\draw [in=120, out=75] (14.center) to (13.center);
				\draw (8) to (2.center);
				\draw (8) to (15.center);
				\draw [in=-150, out=30, looseness=1.25] (26.center) to (25.center);
				\draw [in=-150, out=30, looseness=1.25,dashed] (27.center) to (28.center);
			\end{pgfonlayer}
		\end{tikzpicture}

		\label{fig:enter-label}
	\end{figure}

	Combining \eqref{conclusion_small_time_ctrl} and \eqref{time_continty} by stability property Proposition \ref{prp_continty} and Lemma \ref{alg_r(t)} we obtain,
	\begin{align}\label{on_off_ctrl_step1}
		\| \mathcal{R}_{\delta_1 + t}(u_0, 0, \eta) - u_1 \|_s \le \varepsilon \ \text{ for all } t \in [0, \tau],
	\end{align}
	where the $\mathcal{H}(I)$- valued control $\eta$ is defined as $$\eta : s \to \mathbbm{1}_{[0, \delta_1]} \hat{\eta}_1(s) + \mathbbm{1}_{(\delta_1, \delta_1 + \tau]} 0.$$
	If $\delta_1 + \tau > T,$ then the proof is done. Otherwise we continue the trajectory, first by small time controllability Proposition \ref{prp_small_time_ctrl} to reach in the ball $B_{H^s_0(\mathbb{T})}(u_1, r)$ and then apply \eqref{time_continty}. So using Lemma \ref{alg_r(t)}, we deduce that
	\begin{align}
		\| \mathcal{R}_{\delta_1 + t + \delta_2 + t}(u_0, 0, \tilde{\eta}) - u_1\|_s \le \varepsilon \ \text{ for all } t \in [0, \tau],
	\end{align}
	for some $\delta_2 > 0$ and $\tilde{\eta}$ be $\mathcal{H}(I)$- valued control.
	
	If $\delta_1 + \tau + \delta_2 + \tau > T$ then the proof is complete. Otherwise we continue this process and after a finite number of steps the control trajectory reach near to $u_1$ at time $T.$
	
	\noindent \textbf{Necessity condition for controllability :} Now assume, $I \subset \mathbb{Z^*}$ be a finite symmetric set, which is not a generator of $\mathbb{Z},$ then $$\gcd (i, \cdots, j) = d,$$ where $i, \cdots, j \in I$ and $d $ is some integer strictly bigger than $1.$ This implies, $$\bigcup\limits_{n \in \mathbb{N}} I_n = d\mathbb{Z}.$$  Thus, it is clear that, 
	\begin{align}\label{complement_element}
		\Big\{\cos ((d + 1)x), \sin ((d + 1)x)\Big\} \text{ is orthogonal to }  \mathcal{H}_{\infty}^{I}.
	\end{align}
	The reachable set from $u_0$ at time $T$ defined by,
	\begin{align}\label{reachable_set}
		\mathcal{A}_{u_0, T} := \Big\{ \mathcal{R}_T(u_0, 0, \eta) : \eta \in L^2(0, T; \mathcal{H}(I))\Big\}.
	\end{align}
	Now, let initial data $u_0 $ in $\mathcal{H}_{\infty}^{I}.$ As the space $\mathcal{H}_{\infty}^{I}$ is invariant under both linear dynamics and nonlinear term $u \partial_{x}u.$ So $\mathcal{A}_{u_0, T}$ is orthogonal to  the functions $\cos ((d + 1)x)$ and $ \sin ((d + 1)x),$ This implies $\mathcal{A}_{u_0, T}$ is not dense in $H^s(\mathbb{T}).$  Hence we can't have approximate controllability in $H^s(\mathbb{T}).$
\end{proof}

It remains to provide the proof of Proposition \ref{prp_small_time_ctrl}. Let us now present the formal proof.

\begin{proof}[Proof of Proposition \ref{prp_small_time_ctrl}]
Let us for the moment $u_0 \in H_0^{s + 7}(\mathbb{T})$,  $u_1 \in H^s_0(\mathbb{T})$ and $I_0 \subset \mathbb{Z^*}$, which is symmetric and generator of $\mathbb{Z},$ then by the Lemma \ref{lema_regarding_I_n}, $\mathcal{H}_{\infty}$ is dense in $H^s_0(\mathbb{T}),$ so for all $\varepsilon > 0,$ there exits $n\in \mathbb{N}$ and  $\tilde{u}_1 \in u_0 + \mathcal{H}(I_{n + 1})$ such that $$\| u_1 - \tilde{u}_1\|_s \le \varepsilon.$$ Moreover, in rest of the proof, we will call $w$ the element which belongs to $\mathcal{H}(I_{n + 1}) \subset \mathcal{C}(I_n),$ existence of such $w$ clear from second part of Lemma \ref{lema_regarding_I_n}.

To establish small-time global approximate controllability in \( H^s_0(\mathbb{T}) \), we proceed as follows:  

\begin{itemize}  
	\item[\underline{STEP 1}:] For any \( n \in \mathbb{N} \), we demonstrate that it is possible to reach arbitrarily close to \( u_0 + w \) from \( u_0 \) in a small time, for any \( u_0 \in H^{s + 7}_0(\mathbb{T}) \) and any \( w \in \mathcal{H}(I_{n + 1}) \subset \mathcal{C}(I_n) \).  
	
	\item[\underline{STEP 2}:] By leveraging the density of \( H_0^{s + 7}(\mathbb{T}) \) in \( H_0^s(\mathbb{T}) \) and the stability properties provided by Proposition \ref{prp_continty}, we extend the result to establish small-time global approximate controllability for \( u_0 \in H_0^s(\mathbb{T}) \).  
\end{itemize}  
Thus, to complete the proof, it suffices to establish Step 1.
  
For  $ I \subset \mathbb{Z}$ we first define the sets : 
\begin{align*}
\begin{cases}
&\mathcal{C}_k(I) := \Bigg\{ \eta - \sum_{i = 1}^{k} \zeta_i \partial_{x} \zeta_i \ ; \ \eta, \varphi_i \in \mathcal{H}(I) \Bigg\}  \ \text{ when }  k \in \mathbb{N^*},\\
& \mathcal{C}_0(I) := \mathcal{H}(I).
\end{cases}
\end{align*}
It is clear that for any $k \in \mathbb{N},$ $$\mathcal{C}_k(I) \subseteq \mathcal{C}(I), \ \text{ for any } I \subset \mathbb{Z}.$$
Now we prove the following property by double induction on both index $k$ and $n,$ 
\begin{align*}
\mathbb{P}_{n, k} \hspace{.5cm} &:\ \forall u_0 \in H^{s + 7}(\mathbb{T}), \forall \varepsilon > 0, \forall \sigma > 0, \forall w \in \mathcal{C}_k(I_n), \exists\ t \in(0,\sigma), \text{ and }  \exists \ \eta \in L^2(0, t;  \mathcal{H}(I_0)) \\
& \hspace{2cm}\textit{ such that }\ \| \mathcal{R}_t(u_0, 0, \eta) - (u_0 + w)\|_{s} \le \varepsilon.
\end{align*}

\noindent \textbf{Base case:} In this we will prove the property $\mathbb{P}_{0,0}.$ Let $u_0 \in H^{s + 7}, \ \varepsilon, \sigma > 0$ and $w \in \mathcal{C}_0(I_0)$ are given. Then $$w = \eta, \quad \text{ for some } \eta \in \mathcal{H}(I_0).$$ Now by the Proposition \ref{prp_limit}, for the couple $(0, \delta^{-1} \eta),$ we have $$\mathcal{R}_{\delta}(u_0, 0, \delta^{-1}\eta) \to u_0 + \eta, \ \text{ in } H^s(\mathbb{T}) \ \text{ as } \delta \to 0^{+}.$$
From the above limit we can found small enough $\delta_1 \in (0, \sigma)$ and a control $\bar{\eta} = \delta_1^{-1}\eta \in L^2(0, \delta_1; \mathcal{H}(I_0))$such that $$ \| \mathcal{R}_{\delta_1}(u_0, 0, \bar{\eta}) - (u_0 + w)\|_s \le \varepsilon.$$

\noindent \textbf{Inductive case:} Assume that $\mathbb{P}_{n, k}$ holds for some $(n, k) \in \mathbb{N} \times \mathbb{N^*}.$ Similarly like base case, first fix $u_0 \in H^{s + 7}, \ \varepsilon, \sigma > 0$ and $w \in \mathcal{C}_k(I_n),$ then $$v := w - \zeta \partial_{x} \zeta \in \mathcal{C}_{k + 1}(I_n),$$ where $\zeta \in \mathcal{H}(I_{n}).$

\noindent Now, using Asymptotic property \ref{prp_limit} and flow property \ref{uniqueness_property}, we have 
\begin{align*}
\mathcal{R}_{\delta}(u_0 + \delta^{-\frac{1}{2}} \zeta, 0, 0) -  \delta^{-\frac{1}{2}} \zeta \to (u_0 - \zeta \partial_{x} \zeta) \ \text{ in } H^s, \ \text{ as } \delta \to 0^+.
\end{align*}
Consequently, there exists $\theta_1 \in (0, \frac{\sigma}{3}),$ such that 
\begin{align}\label{ineq_1}
\| \mathcal{R}_{\theta_1}(u_0 + w + \theta^{-\frac{1}{2}}_1 \zeta, 0, 0 ) - (u_0 + w - \zeta \partial_{x} \zeta + \theta^{-\frac{1}{2}}_1 \zeta)\|_s \le \varepsilon.
\end{align}
Now, $\theta_1^{-\frac{1}{2}}$ is a fixed positve constant and $\zeta \in \mathcal{H}(I_n)$ so $w +  \theta^{-\frac{1}{2}}_1 \zeta \in \mathcal{C}_k(I_n),$ then by induction hypothesis there esists $\tau_1 \in (0, \frac{\sigma}{3}),$ and a control $\eta_1 \in L^2(0, \tau_1; \mathcal{H}(I_0))$ we have
\begin{align}\label{ineq_2}
\| \mathcal{R}_{\tau_1}(u_0, 0, \eta_1) - (u_0 + w +  \theta^{-\frac{1}{2}}_1 \zeta )\|_s \le \varepsilon.
\end{align}
Thus, by the Lemma \ref{alg_r(t)}, starting from $u_0$ we can reach close $u_0 + w - \zeta \partial_x \zeta + \theta^{-\frac{1}{2}}_1 \zeta,$ at time $\tau_1 + \theta_1$ by the control $$\eta : s \to \mathbbm{1}_{[0, \tau_1]} \eta_1(s) + \mathbbm{1}_{(\tau_1, \tau_1 + \theta_1]} 0,$$ since we have
\begin{align}\label{ineq_3}
&\Bigg\|\mathcal{R}_{ \tau_1 + \theta_1 }(u_0, 0, \eta) - (u_0 +  w - \zeta \partial_x \zeta  +  \theta^{-\frac{1}{2}}_1 \zeta ) \Bigg\|_{s}
 \le \Bigg\| \mathcal{R}_{\theta_1}(\mathcal{R}_{\tau_1}(u_0, 0, \eta_1), 0, 0) - \mathcal{R}_{\theta_1}(u_0 + w + \theta_1^{-\frac{1}{2}} \zeta, 0, 0)\Bigg\|_{s} \notag\\
& \hspace{4cm} + \Bigg\| \mathcal{R}_{\theta_1}(u_0 + w + \theta_1^{-\frac{1}{2}}\zeta, 0, 0) - (u_0 + w - \zeta \partial_x \zeta  + \theta_{1}^{-\frac{1}{2}}\zeta )\Bigg\|_{s} \le \varepsilon,
\end{align}
where for estimate of the first term we have used \eqref{ineq_2}, Proposition \ref{prp_continty} of the resolvent map $\mathcal{R}_{\theta_1}$ and  \eqref{ineq_1} for the second term.

As, $w$ and  $\zeta,$ are smooth enough and,  by induction hypothesis $\mathbb{P}_{n,k}$ is true for any $u_0 \in H^{s+7}(\mathbb{T})$ so, we can choose initial data $\hat{u}_0 = u_0 + w - \zeta \partial_{x} \zeta + \theta_1^{-\frac{1}{2}}\zeta.$ Then by induction hypothesis there exists a time $\tau_2 \in (0, \frac{\sigma}{3})$ and a control $\eta_2 \in L^2(0, \tau_2; \mathcal{H}(I_0))$ such that 
\begin{align}\label{ineq_4}
\| \mathcal{R}_{\tau_2}(\hat{u}_0, 0, \eta_2) - (\hat{u}_0 - \theta_1^{-\frac{1}{2}}\zeta) \|_s \le \varepsilon.
\end{align} 
Now by the control $\eta $ defined as, $$\eta : s \to \mathbbm{1}_{[0, \tau_1]} \eta_1(s) + \mathbbm{1}_{(\tau_1, \tau_1 + \theta_1]} 0 + \mathbbm{1}_{(\tau_1 + \theta_1, \tau_1 + \theta_1 +\tau_2] } \eta_2(s - \tau_1 - \theta_1),$$
using Lemma \ref{alg_r(t)}, \eqref{ineq_3} , \eqref{ineq_4} and the stability Proposition \ref{prp_continty} we have the following estimate
\begin{align}\label{ineq_5}
\| \mathcal{R}_{\tau_2 + \tau_1 + \theta_1}(u_0, 0, \eta) - (u_0 + v)\|_s \le \varepsilon.
\end{align}
Let $t := \tau_1 + \theta_1 + \tau_2 \in (0, \sigma)$ then we conclude the result holds for $\mathbb{P}_{n, k+1}.$

\noindent Thus, assuming the property for $\mathbb{P}_{n,k},$ we are able to deduce for $\mathbb{P}_{n, k+1},$ where $(n, k) \in \mathbb{N} \times \mathbb{N^*},$ althoughit is not enough, so for concluding the property for $(n, k) \in \mathbb{N} \times \mathbb{N},$ our next aims are the following :
\begin{enumerate}[(a)]
\item  $\mathbb{P}_{n, 0}$ is holds for any $n \in \mathbb{N}.$ \label{n,0_ind}
\item  For $n \in \mathbb{N},$ if $\mathbb{P}_{n, 0} $ is holds that implies the property is true for $\mathbb{P}_{n, 1}.$ 
\end{enumerate} 
To complete the argument of \eqref{n,0_ind}, we again use indution on $n.$ The base case $\mathbb{P}_{0, 0}$ is already done. Now assume the property holds for $\mathbb{P}_{n, 0}.$ So for any $\zeta \in \mathcal{H}(I_n)$ we can reach to $u_0 + \zeta$ from $u_0$ in small time by $\mathcal{H}(I_0)$ valued control.

 Let $w \in \mathcal{H}(I_{n + 1}) $ then by Lemma \ref{lema_regarding_I_n}, $w \in \mathcal{C}(I_n),$ so $$w \in  \mathcal{C}_k(I_n), \ \text{ for some } k.$$ Now we can repete the analysis, like \eqref{ineq_1} - \eqref{ineq_5} to conclude the property $\mathbb{P}_{n+1, 0}.$ Hence $\mathbb{P}_{n, 0}$ is holds for all $n \in \mathbb{N}.$
 
 For the argument (ii), choose $M \in \mathbb{N},$ and using (i), the property $\mathbb{P}_{M, 0}$ holds. Let $\zeta \in C_1(I_M)$ then $$\zeta = \zeta_1 - \zeta_2 \partial_{x} \zeta_2,$$ for some $\zeta_1, \zeta_2 \in \mathcal{H}(I_M).$ As $\mathbb{P}_{M, 0}$ is true then again by saimilar resoning, like \eqref{ineq_1} - \eqref{ineq_5} to obtain the desire result. Hence $\mathbb{P}_{n, k}$ is true for all $(n, k) \in \mathbb{N} \times \mathbb{N}.$

\end{proof}

\begin{remark}
In the statement of Theorem \ref{main_thm}, the symmetry condition on \( I \) is necessary. This is because there exist generating sets of the form  
\[  
I = \{a_1, a_2 \in \mathbb{Z}^+ : \gcd(a_1, a_2) = 1\},  
\]  
which generate \( \mathbb{Z} \), but for which  
\[  
\bigcup_{n \in \mathbb{N}} I_n \neq \mathbb{Z}^+.  
\]  
\end{remark}

%

\section{Proof of Proposition \ref{prp_existance} - \ref{prp_limit}.}\label{sec_prop} 
In this section, unless specified otherwise, \( C \) denotes a generic positive constant whose value may vary from line to line. When necessary, the dependence of \( C \) on certain parameters, denoted by $`` \cdot ",$ will be explicitly written as \( C(\cdot) \).
\subsection{The Bourgain Space and Its Properties}
Here, we present some tools and results that are essential for proving Propositions \ref{prp_existance}–\ref{prp_limit}. Recall that \( a = [u_0(x)] = \frac{1}{2\pi} \int_{\mathbb{T}} u_0(x) \, dx \). Next, we define the unbounded operator \(\left( \mathcal{A}, \mathcal{D}(\mathcal{A}; L^2(\mathbb{T})) \right)\) in \(L^2(\mathbb{T})\) as follows: 

\[
\mathcal{D}(\mathcal{A}; L^2(\mathbb{T})) = H^5(\mathbb{T}),
\]
\[
\mathcal{A} = \frac{d^5}{dx^5} - \frac{d^3}{dx^3} - a \frac{d}{dx}.
\]

The adjoint \(\mathcal{A}^*\) of \(\mathcal{A}\), with the domain \(\mathcal{D}(\mathcal{A}^*; L^2(\mathbb{T})) = \mathcal{D}(\mathcal{A}; L^2(\mathbb{T}))\), is given by:
\[
\mathcal{A}^* = -\frac{d^5}{dx^5} + \frac{d^3}{dx^3} + a \frac{d}{dx} = -\mathcal{A}.
\]

Thus, \(\mathcal{A}\) is a skew-adjoint operator. By the Stone Theorem (\cite[Theorem 3.8.6]{Tucsnak_book}), \(\mathcal{A}\) generates a strongly continuous unitary group \(\left\{ \mathcal{W}(t) \right\}\) on \(L^2(\mathbb{T})\).

The eigenfunctions of \(\left( \mathcal{A}, \mathcal{D}(\mathcal{A}; L^2(\mathbb{T})) \right)\) are the orthonormal Fourier basis functions in \(L^2(\mathbb{T})\):
\[
\phi_k(x) = \frac{1}{\sqrt{2\pi}} e^{ikx}, \quad k \in \mathbb{Z}.
\]

The corresponding eigenvalue of \(\phi_k\) is:
\[
\lambda_k = \left( k^5 + k^3 - ak \right)i, \quad k \in \mathbb{Z}.
\]
We now introduce the Bourgain space, which will be instrumental in proving Proposition \ref{prp_existance}. For given \( s, b \in \mathbb{R} \) and a function \( u: \mathbb{R} \times \mathbb{T} \rightarrow \mathbb{R} \), we define the norms 

\[
\| u \|_{X_{b,s}} := \left( \sum_{k = -\infty}^{\infty} \int_{\mathbb{R}} \langle k \rangle^{2s} \langle \tau - k^5 - k^3 + ak \rangle^{2b} | \hat{u}(\tau, k) |^2 \, d\tau \right)^{\frac{1}{2}},
\]

and 

\[
\| u \|_{Y_{b,s}} := \left( \sum_{k = -\infty}^{\infty} \left( \int_{\mathbb{R}} \langle k \rangle^{s} \langle \tau - k^5 - k^3 + ak \rangle^{b} | \hat{u}(\tau, k) | \, d\tau \right)^2 \right)^{\frac{1}{2}},
\]

where \( \langle \cdot \rangle = \sqrt{1 + | \cdot |^2} \), and \( \hat{u}(\tau, k) \) denotes the Fourier transform of \( u \) with respect to the time variable \( t \) and the spatial variable \( x \). 

The Bourgain space \( X_{b,s} \) (resp. \( Y_{b,s} \)) associated with the Kawahara equation on \( \mathbb{T} \) is defined as the completion of the Schwartz space \( \mathcal{S}(\mathbb{R} \times \mathbb{T}) \) under the norm \( \| u \|_{X_{b,s}} \) (resp. \( \| u \|_{Y_{b,s}} \)). We also define the space 

\[
Z_{b,s} = X_{b,s} \cap Y_{b-\frac{1}{2},s},
\]

equipped with the norm 

\[
\| u \|_{Z_{b,s}} = \| u \|_{X_{b,s}} + \| u \|_{Y_{b-\frac{1}{2},s}}.
\]

For a given time interval \( I \), we define the restriction spaces \( X_{b,s}(I) \) (resp. \( Z_{b,s}(I) \)) as 

\[
\| u \|_{X_{b,s}(I)} = \inf \left\{ \| \tilde{u} \|_{X_{b,s}} \, \middle| \, \tilde{u} = u \text{ on } I \times \mathbb{T} \right\},
\]

\[
\text{(resp. } \| u \|_{Z_{b,s}(I)} = \inf \left\{ \| \tilde{u} \|_{Z_{b,s}} \, \middle| \, \tilde{u} = u \text{ on } I \times \mathbb{T} \right\} \text{)}.
\]

For simplicity, we denote \( X_{b,s}(I) \) (resp. \( Z_{b,s}(I) \)) by \( X^T_{b,s} \) (resp. \( Z^T_{b,s} \)) when \( I = (0, T) \).

We now gather some properties of the Bourgain spaces \( X^T_{b,s} \) and \( Z^T_{b,s} \), which will play a crucial role in establishing the controllability of the nonlinear Kawahara equation.
\begin{lemma}\label{Bourgain_space}
	Let $ s \ge 0,$ $ T > 0$, and $ I $ be a given time interval. The following properties hold:
	\begin{enumerate}
		\item \label{Bourgain_space_inclusion}  If $b_1 \le b_2$ and $s_1 \le s_2,$ then $ X_{b_2, s_2}$is continuously embedded in $X_{b_1, s_1},$ i.e., $ X_{b_2, s_2}  \lhook\joinrel\xrightarrow{\text{Cts}} X_{b_1, s_1}.$
		\item \label{Bourgain_space_in_continuous}  
		\( Z_{\frac{1}{2}, s}(I) \lhook\joinrel\xrightarrow{\text{Cts}} C(\overline{I}; H^s(\mathbb{T})) \) for any \( s \in \mathbb{R} \).  
		
		\item \label{semigrp_est}  
		There exists a constant \( C = C(s, T) > 0 \) such that:  
		\begin{enumerate}  
			\item[(i)] For any \( \phi \in H^s(\mathbb{T}) \),  
			\[    \| \mathcal{W}(\cdot) \phi \|_{X^T_{\frac{1}{2}, s}} \leq C \|\phi\|_s, \quad      \| \mathcal{W}(\cdot) \phi \|_{Z^T_{\frac{1}{2}, s}} \leq C \|\phi\|_s.;     \]  
			
			\item[(ii)] for any \( f \in Z^T_{-\frac{1}{2}, s} \),  
			\[    \left\| \int_0^{\cdot} \mathcal{W}(\cdot - \tau) f(\tau) \, d\tau \right\|_{Z^T_{\frac{1}{2}, s}} \leq C \|f\|_{Z^T_{-\frac{1}{2}, s}}.      \]  
		\end{enumerate}  
		
		\item \label{bilinear_est}  
		Let \( \theta \in (0, 1) \) and \( u, v \in X^\theta_{\frac{1}{2}, s} \cap L^2(0, \theta; L_0^2(\mathbb{T})) \). Then there exist constants \( \alpha > 0 \) and \( C > 0 \), independent of \( \theta \), \( u \), and \( v \), such that  
		\[\| (uv)_x \|_{Z^\theta_{-\frac{1}{2}, s}} \leq C \theta^\alpha \|u\|_{X^\theta_{\frac{1}{2}, s}} \|v\|_{X^\theta_{\frac{1}{2}, s}} \leq C \theta^\alpha \|u\|_{Z^\theta_{\frac{1}{2}, s}} \|v\|_{Z^\theta_{\frac{1}{2}, s}}.  \]  
		
		\item \label{L^2_Bourgain_space}  
		For any \( u \in L^2(0, T; H^s(\mathbb{T})) \), there exists a constant \( C = C(s, T) > 0 \) such that  
		\[\| u \|_{Z^T_{-\frac{1}{2}, s}} \leq C \| u \|_{L^2(0, T; H^s(\mathbb{T}))}.  \]  
		
		\item \label{H_s+5_in_Bourgain}  
		For any \( u \in H^{s+5}(\mathbb{T}) \), there exists a constant \( C = C(s, T) > 0 \) such that  
		\[\| u \|_{X^T_{\frac{1}{2}, s}} \leq C \| u \|_{{s+5}}.  \]  
	\end{enumerate}
\end{lemma}
\begin{proof}  
	The proofs of properties \eqref{Bourgain_space_inclusion}–\eqref{bilinear_est} can be found in \cite{Bourgain_93, Hirayama_12, Tao_06}. Property \eqref{L^2_Bourgain_space} follows directly from the definition of \( Z^{T}_{-\frac{1}{2}, s} \). Therefore, it remains to establish property \eqref{H_s+5_in_Bourgain}.  
	
	Let \(\beta \in C^{\infty}_0 (\mathbb{R})\) be a smooth function such that  
	\[
	\beta(t) = 
	\begin{cases} 
		1 & \text{for } 0 \leq t \leq T, \\ 
		0 & \text{for } t \leq -1 \text{ or } t \geq T+1.  
	\end{cases}
	\]  
	
	Then, for any \(u \in X^T_{\frac{1}{2}, s}\), we have \(\beta u \in X^T_{\frac{1}{2}, s}\) (see Lemma 2.11 in \cite{Tao_06} for details).  
	
	Using property \eqref{Bourgain_space_inclusion} and the definition of \(X_{1,s}\), we obtain  
	\begin{align*}
		\|u\|^2_{X^T_{\frac{1}{2}, s}} & \le \| \beta u\|^2_{X_{\frac{1}{2}, s}} \\
		& \le C \|\beta u \|^2_{X_{1, s}} \\
		& = C \Big( \|\beta u\|^2_{L^2(\mathbb{R}; H^s(\mathbb{T}))} + \|(\beta u)_t -(\beta u)_{xxxxx} + (\beta u)_{xxx} + a (\beta u)_{x}\|^2_{L^2(\mathbb{R}; H^s(\mathbb{T}))} \Big) \\
		& \le C \Big( \|\beta u\|^2_{L^2(\mathbb{R}; H^s(\mathbb{T}))} + \|(\beta' u)\|^2_{L^2(\mathbb{R}; H^s(\mathbb{T}))} + \|(\beta u)_{xxxxx}\|^2_{L^2(\mathbb{R}; H^s(\mathbb{T}))}\\
		& \hspace{5cm}+ \|(\beta u)_{xxx}\|^2_{L^2(\mathbb{R}; H^s(\mathbb{T}))} + \| a(\beta u)_{x}\|^2_{L^2(\mathbb{R}; H^s(\mathbb{T}))} \Big) \\
		& \le C \|u \|^2_{s+5}.
	\end{align*}
\end{proof}
Now, we are in a position to prove Proposition \ref{prp_existance}.
\subsection{Proof of Proposition \ref{prp_existance}}
\begin{proof}
	Let \(\Tilde{u} = u - a\), where \(a = [u_0]\). If \(u\) satisfies \eqref{xtd_equn}, then \(\Tilde{u}\) satisfies the following equation:  
	
	\begin{align}\label{u_tilde_equn}
		\begin{cases}  
			\Tilde{u}_t - \Tilde{u}_{xxxxx} + \Tilde{u}_{xxx} + a \Tilde{u}_x + (\Tilde{u} \zeta)_x + \Tilde{u} \Tilde{u}_x = \varphi + \zeta_{xxxxx} - \zeta_{xxx} - a \zeta_x - \zeta \zeta_x, & \text{in } (0, T) \times \mathbb{T}, \\  
			\Tilde{u}(0) = \Tilde{u}_0 := u_0 - a, & \text{in } \mathbb{T}.  
		\end{cases}  
	\end{align}  
	
	We can see that $\left[ \Tilde{u}\right] =0$. Let \(\theta \in (0, \min\{1, T\})\) and \(v \in Z^{\theta}_{\frac{1}{2}, s} \cap L^2(0, \theta; L^2_0(\mathbb{T}))\). Define  
	\[
	\Phi(v)(t) = \mathcal{W}(t) \Tilde{u}_0 + \int_0^t \mathcal{W}(t - s) \Big( -v v_x - (v \zeta)_x + \varphi + \zeta_{xxxxx} - \zeta_{xxx} - a \zeta_x - \zeta \zeta_x \Big)(s) \, ds,  
	\]  
	where \(\{\mathcal{W}(t)\}\) is the strongly continuous unitary group on \(L^2(\mathbb{T})\) generated by the linear operator \(\mathcal{A}\). Then using Lemma \ref{Bourgain_space}, we get
	\begin{align*}
		\| \Phi (v)\|_{Z^{\theta}_{\frac{1}{2}, s}} & \le \| \mathcal{W}( \cdot ) \Tilde{u}_0 \|_{Z^{T}_{\frac{1}{2}, s}} + \left\|\int_0^{\cdot} \mathcal{W}(\cdot - s) \Big( - v v_x - (v \zeta)_x - \zeta \zeta_x \Big) (s) ds \right\|_{Z^{\theta}_{\frac{1}{2}, s}} \\
		&\hspace{2cm} + \left\|\int_0^{\cdot} \mathcal{W}(\cdot - s) \Big( \varphi + \zeta_{xxxxx} - \zeta_{xxx} - a \zeta_x\Big) (s) ds \right\|_{Z^{T}_{\frac{1}{2}, s}}\\
		& \le C \Bigg(\|\tilde{u}_0\|_s +\|(v^2)_x\|_{Z^{\theta}_{-\frac{1}{2}, s}} + \|(v \zeta)_x\|_{Z^{\theta}_{-\frac{1}{2}, s}}  + \| (\zeta^2)_x \|_{Z^{\theta}_{-\frac{1}{2}, s}}+ \| \varphi \|_{Z^{T}_{-\frac{1}{2}, s}}\\&\hspace{2cm} + \| \zeta_{xxxxx} \|_{Z^{T}_{-\frac{1}{2}, s}} + \| \zeta_{xxx} \|_{Z^{T}_{-\frac{1}{2}, s}} + \| \zeta_x \|_{Z^{T}_{-\frac{1}{2}, s}} \Bigg)\\
		& \le C \Bigg(\|\tilde{u}_0\|_s + \| \varphi \|_{Z^{T}_{-\frac{1}{2}, s}} + \| \zeta_{xxxxx} \|_{Z^{T}_{-\frac{1}{2}, s}} + \| \zeta_{xxx} \|_{Z^{T}_{-\frac{1}{2}, s}} + \| \zeta_x \|_{Z^{T}_{-\frac{1}{2}, s}} \Bigg)\\
		& \hspace{2cm} + C \theta^{\alpha} \Bigg( \|v\|^2_{Z^{\theta}_{\frac{1}{2}, s}} + \|v\|_{Z^{\theta}_{\frac{1}{2}, s}} \| \zeta\|_{X^{\theta}_{\frac{1}{2}, s}} + \| \zeta \|^2_{X^{\theta}_{\frac{1}{2}, s}} \Bigg)\\
		& \le  C \Bigg( \|\Tilde{u}_0\|_s + \|\varphi \|_{Z^{T}_{-\frac{1}{2}, s}} + \| \zeta_{xxxxx} \|_{Z^{T}_{-\frac{1}{2}, s}} + \|\zeta_{xxx} \|_{Z^{T}_{-\frac{1}{2}, s}}  + \| \zeta_x \|_{Z^{T}_{-\frac{1}{2}, s}}+ \|\zeta \|^2_{X^{T}_{\frac{1}{2}, s}}\Bigg) + C \theta^{\alpha} \|v\|^2_{Z^{\theta}_{\frac{1}{2}, s}}\\
		& \le C_1 \Big( 1 + \|u_0\|_s + \|\varphi\|_{L^2(0, T; H^s(\mathbb{T}))} + \| \zeta \|^2_{s + 5} \Big) + C_2 \theta^{\alpha} \| v\|^2_{{Z^{\theta}_{\frac{1}{2}, s}}},
	\end{align*}  
	where \(C_1, C_2\) are constants independent of \(\theta\).   For any \(v_1, v_2 \in Z^{\theta}_{\frac{1}{2}, s} \cap L^2(0, \theta; L^2_0(\mathbb{T}))\), we can similarly derive the following:  
	
	\[
	\| \Phi(v_1) - \Phi(v_2) \|_{Z^{\theta}_{\frac{1}{2}, s}} \leq C \| \big((v_1 + v_2)(v_1 - v_2)\big)_x \|_{Z^{\theta}_{-\frac{1}{2}, s}} + C \| \big((v_1 - v_2)\zeta\big)_x \|_{Z^{\theta}_{-\frac{1}{2}, s}}.  
	\]  
	Applying properties \ref{bilinear_est} and \ref{H_s+5_in_Bourgain} of Lemma \ref{Bourgain_space},  
	\begin{align*}
		\| \Phi(v_1) - \Phi(v_2) \|_{Z^{\theta}_{\frac{1}{2}, s}} \leq &C \theta^{\alpha} \Big( \|v_1\|_{Z^{\theta}_{\frac{1}{2}, s}} + \|v_2\|_{Z^{\theta}_{\frac{1}{2}, s}} + \|\zeta\|_{X^{T}_{\frac{1}{2}, s}} \Big) \|v_1 - v_2\|_{Z^{\theta}_{\frac{1}{2}, s}}. \\
		\leq &C_3 \theta^{\alpha} \Big( \|v_1\|_{Z^{\theta}_{\frac{1}{2}, s}} + \|v_2\|_{Z^{\theta}_{\frac{1}{2}, s}} + \|\zeta\|_{s + 5} \Big) \|v_1 - v_2\|_{Z^{\theta}_{\frac{1}{2}, s}}, 
	\end{align*}
	
	where \(C_3\) is independent of \(\theta\).  
	
	Define  
	\[
	R = 2 C_1 \Big(1 + \|u_0\|_s + \|\varphi\|_{L^2(0, T; H^s(\mathbb{T}))} + \|\zeta\|^2_{s + 5}\Big).  
	\]  
	
	For \(v, v_1, v_2 \in B_{Z^{\theta}_{\frac{1}{2}, s}}(R) := \left\{ u \in Z^{\theta}_{\frac{1}{2}, s} \cap L^2(0, \theta; L^2_0(\mathbb{T})); \|u\|_{Z^{\theta}_{\frac{1}{2}, s}} \leq R \right\}\), we have  
	\[
	\|\Phi(v)\|_{Z^{\theta}_{\frac{1}{2}, s}} \leq \frac{R}{2} + C_2 \theta^{\alpha} R^2,  
	\]  
	\[
	\|\Phi(v_1) - \Phi(v_2)\|_{Z^{\theta}_{\frac{1}{2}, s}} \leq C_3 \theta^{\alpha} \Big(2R + \|\zeta\|_{s + 5}\Big) \|v_1 - v_2\|_{Z^{\theta}_{\frac{1}{2}, s}}.  
	\]  
	
	Choose \(\theta\) small enough such that  
	\[
	\theta^{\alpha} \max\{C_2 R, C_3(2R + \|\zeta\|_{s + 5})\} \leq \frac{1}{2}.  
	\]  
	Then,  
	\[
	\|\Phi(v)\|_{Z^{\theta}_{\frac{1}{2}, s}} \leq R, \quad \|\Phi(v_1) - \Phi(v_2)\|_{Z^{\theta}_{\frac{1}{2}, s}} \leq \frac{1}{2} \|v_1 - v_2\|_{Z^{\theta}_{\frac{1}{2}, s}}.  
	\]  
	
	By the fixed-point theorem, \(\Phi\) has a unique fixed point \(\Tilde{u} \in B_{Z^{\theta}_{\frac{1}{2}, s}}(R)\). Consequently, by property \eqref{Bourgain_space_in_continuous} from Lemma \ref{Bourgain_space}, the system \eqref{u_tilde_equn} has a unique solution \(\Tilde{u} \in C([0, \theta]; H_0^s(\mathbb{T}))\) for sufficiently small \(\theta\). Equivalently, the system \eqref{xtd_equn} has a unique solution \(u \in C([0, \theta]; H^s(\mathbb{T}))\) for small \(\theta\).


	To extend this local solution globally, we have to establish the following a priori estimate:  
	\begin{align}\label{a_priori_est}  
		\sup_{t \in [0, T]} \|u(t)\|_s \leq C, \quad \forall s \in \mathbb{N},  
	\end{align}  
	where \(C > 0\) depends on \(T, s, \|u_0\|_s, \|\zeta\|_{s + 7}\), and \(\|\varphi\|_{L^2(0, T; H^s(\mathbb{T}))}\).  
	
	Case \(s = 0\):   Multiplying \eqref{xtd_equn} by \(2u\) and integrating over \(\mathbb{T}\), we obtain  
	\[
	\frac{d}{dt} \| u\|^2 \leq C \Big( \|\zeta\|_5 \|u\| + \|\zeta\|_3 \|u\| + \|\zeta\|_5 \|u\|^2 + \|\zeta\|_5^2 \|u\| + \|\varphi\| \|u\| \Big).
	\]  
	Using the fact that \(\zeta \in H^5(\mathbb{T})\) and \(\varphi \in L^2(0, T; L^2_0(\mathbb{T}))\), and applying the inequality \(ab \leq \frac{a^2}{2} + \frac{b^2}{2}\), we have  
	\[
	\frac{d}{dt} \|u\|^2 \leq C_1 \Big(1 + \|\zeta\|_5 \Big) \|u\|^2 + C_2.
	\]  
	Applying Grönwall's inequality gives  
	\begin{align}\label{S=0_inequality}
		\sup_{t \in [0, T]} \|u(t)\|^2 \leq C,
	\end{align}  
	where \(C > 0\) depends on \(T\), \(\|\zeta\|_5\), \(\|\varphi\|_{L^2(0, T; L^2_0(\mathbb{T}))}\), and \(\|u_0\|\).

	Case \(s = 2\):   Multiplying \eqref{xtd_equn} with \( \mathcal{I} := u^2 + 2u_{xx} - 2u_{xxxx} \), we obtain  
	\[
	\begin{aligned}
		&\Big\langle \mathcal{I}, -u_t \Big\rangle = \frac{d}{dt} \left(- \frac{1}{3} \int_{\mathbb{T}} u^3 \, dx + \| u_x(t)\|^2 + \| u_{xx}(t)\|^2 \right), \\
		&\Big\langle \mathcal{I}, - (u + \zeta)_{xxxxx} \Big\rangle = 2 \int_{\mathbb{T}} u u_x u_{xxxx} \, dx - \int_{\mathbb{T}} u^2 \zeta_{xxxxx} \, dx - 2 \int_{\mathbb{T}} u_{xx} \zeta_{xxxxx} \, dx + 2 \int_{\mathbb{T}} u_{xxxx} \zeta_{xxxxx} \, dx, \\
		&\Big\langle \mathcal{I}, (u + \zeta)_{xxx} \Big\rangle = - 2 \int_{\mathbb{T}} u u_x u_{xx} \, dx + \int_{\mathbb{T}} \Big( u^2 + 2 u_{xx} - 2 u_{xxxx} \Big) \zeta_{xxx} \, dx, \\
		&\Big\langle \mathcal{I}, u (u + \zeta)_x \Big\rangle = 2 \int_{\mathbb{T}} u u_x u_{xx} \, dx - 2 \int_{\mathbb{T}} u u_x u_{xxxx} \, dx + \int_{\mathbb{T}} \Big( u^2 + 2 u_{xx} - 2 u_{xxxx} \Big) u \zeta_x \, dx, \\
		&\Big\langle \mathcal{I}, \zeta (u + \zeta)_x \Big\rangle = \int_{\mathbb{T}} \Big( u^2 + 2 u_{xx} - u_{xxxx} \Big) \zeta u_x \, dx + \int_{\mathbb{T}} \Big( u^2 + 2 u_{xx} - u_{xxxx} \Big) \zeta \zeta_x \, dx, \\
		&\Big\langle \mathcal{I}, -\varphi \Big\rangle = \int_{\mathbb{T}} \Big( u_{xxxx} - u^2 - 2 u_{xx} \Big) \varphi \, dx.
	\end{aligned}
	\]
	
	\noindent Combining all the terms above, we deduce:
	\[
	\begin{aligned}
		&\frac{d}{dt} \left(- \frac{1}{3} \int_{\mathbb{T}} u^3 \, dx + \| u_x(t)\|^2 + \| u_{xx}(t)\|^2 \right) \\
		&\quad = - \int_{\mathbb{T}} \Big( u^2 + 2 u_{xx} - 2 u_{xxxx} \Big) \zeta_{xxxxx} \, dx + \int_{\mathbb{T}} \Big( u^2 + 2 u_{xx} - 2 u_{xxxx} \Big) \zeta_{xxx} \, dx \\
		&\qquad + \int_{\mathbb{T}} \Big( u^2 + 2 u_{xx} - 2 u_{xxxx} \Big) u \zeta_x \, dx + \int_{\mathbb{T}} \Big( u^2 + 2 u_{xx} - u_{xxxx} \Big) \zeta u_x \, dx \\
		&\qquad + \int_{\mathbb{T}} \Big( u^2 + 2 u_{xx} - u_{xxxx} \Big) \zeta \zeta_x \, dx + \int_{\mathbb{T}} \Big( u_{xxxx} - u^2 - 2 u_{xx} \Big) \varphi \, dx \\
		&\quad =: \mathcal{J}_1 + \mathcal{J}_2 + \mathcal{J}_3 + \mathcal{J}_4 + \mathcal{J}_5 + \mathcal{J}_6.
	\end{aligned}
	\]
	
	\noindent Using Sobolev, Cauchy–Schwarz, and Poincaré inequalities, we estimate \( |\mathcal{J}_i| \) as follows:
	\[
	\begin{aligned}
		|\mathcal{J}_1| &\lesssim \| \zeta_{xxxxx} \|_{\infty} \|u\|^2 + \| \zeta \|_5 \| u_{xx} \| + \| \zeta \|_7 \| u_{xx} \|, \\
		|\mathcal{J}_2| &\lesssim \| \zeta_{xxx} \|_{\infty} \|u\|^2 + \| \zeta \|_3 \| u_{xx} \| + \| \zeta \|_5 \| u_{xx} \|, \\
		|\mathcal{J}_3| &\lesssim \|u\|_{\infty} \| \zeta_x \|_{\infty} (\|u\|^2 + \| u_{xx} \|) + \| \zeta_{xx} \|_{\infty} \| u_x \| \| u_{xx} \| + \| \zeta_x \| \| u_{xx} \|^2, \\
		|\mathcal{J}_4| &\lesssim \|u\|_{\infty} \| \zeta_x \| \| u_x \|^2 + \| \zeta \|_{\infty} \| u_{xx} \|^2 + (\| \zeta_x \|_{\infty} + \| \zeta_{xx} \|_{\infty}) \| u_{xx} \|^2, \\
		|\mathcal{J}_5| &\lesssim \|u\|_{\infty} \| \zeta \|^2 \| u_x \| + \| \zeta_x \|_{\infty} \| \zeta \| \| u_{xx} \| + \| \zeta_{xxx} \|_{\infty} (\| \zeta_x \| + \| \zeta \|) \| u_{xx} \|, \\
		|\mathcal{J}_6| &\lesssim \|u\|_{\infty} \|u\| \| \varphi \| + \| \varphi \| \| u_{xx} \| + \| \varphi_{xx} \| \| u_{xx} \|.
	\end{aligned}
	\]
	
	\noindent Combining these estimates, we have:
	\[
	\frac{d}{dt} \left(- \frac{1}{3} \int_{\mathbb{T}} u^3 \, dx + \| u_x(t) \|^2 + \| u_{xx}(t) \|^2 \right) \leq C \Big( 1 + \| u_x \|^2 + \| u_{xx} \|^2 + \| \varphi \|_2^2 \Big).
	\]

	Integrating over the interval \([0, t]\) and using \eqref{S=0_inequality} along with the Sobolev inequality, we obtain:  
	\begin{align*}
		\|u_x(t)\|^2 + \|u_{xx}(t)\|^2 & \leq C \Bigg( \|\partial_x u_0\|^2 + \|\partial_{xx} u_0\|^2 + \int_{\mathbb{T}} |u_0|^3 \, dx + \int_{\mathbb{T}} |u(t)|^3 \, dx \\
		& \hspace{2cm} + \int_0^t \Big(1 + \|u_x(s)\|^2 + \|u_{xx}(s)\|^2 + \|\varphi(s)\|_2^2 \Big) \, ds \Bigg) \\
		& \leq C \|u(t)\|_\infty \|u(t)\|^2 + C \Bigg( 1 + \int_0^t \Big( \|u_x(s)\|^2 + \|u_{xx}(s)\|^2 \Big) \, ds \Bigg) \\
		& \leq C \|u_x(t)\| \|u(t)\|^2 + C \Bigg( 1 + \int_0^t \Big( \|u_x(s)\|^2 + \|u_{xx}(s)\|^2 \Big) \, ds \Bigg) \\
		& \leq \frac{1}{2} \|u_x(t)\|^2 + C \|u(t)\|^4 + C \Bigg( 1 + \int_0^t \Big( \|u_x(s)\|^2 + \|u_{xx}(s)\|^2 \Big) \, ds \Bigg).
	\end{align*}
	
	Simplifying further, we deduce:
	\[
	\|u_x(t)\|^2 + \|u_{xx}(t)\|^2 \leq C \Bigg( 1 + \int_0^t \Big( \|u_x(s)\|^2 + \|u_{xx}(s)\|^2 \Big) \, ds \Bigg).
	\]
	
	By applying Gronwall's inequality, it follows that:
	\[
	\|u_x(t)\|^2 + \|u_{xx}(t)\|^2 \leq C, \quad \forall t \in [0, T].
	\]
	
	Finally, combining this result with \eqref{S=0_inequality}, we conclude:
	\begin{align}\label{s=2_ineq}
		\sup_{t \in [0, T]} \|u(t)\|_2^2 \leq C.
	\end{align}
	
	Case \(s = 1\): This case follows from \eqref{S=0_inequality}, \eqref{s=2_ineq} and nonlinear interpolation (see \cite{Tartar_1972, Bergh_1976}).

	Case \(s \geq 3\):   We will establish \eqref{a_priori_est} using mathematical induction. Assume that \eqref{a_priori_est} holds for \(s - 1 \geq 2\). Multiplying \eqref{xtd_equn} by \(\partial_x^{2s} u\), we get  
	\[
	\frac{1}{2} \frac{d}{dt} \|\partial_x^s u\|^2 \leq \|\zeta\|_{s+5} \|\partial_x^s u\| + \frac{1}{2} \left| \left\langle \partial_x^{s+1}(u^2),\ \partial_x^s u \right\rangle \right| + \left| \left\langle \partial_x^{s+1}(\zeta u),\ \partial_x^s u \right\rangle \right| + C \|\zeta\|_{s+1}^2 \|\partial_x^s u\| + \|\varphi\|_s \|\partial_x^s u\|.
	\]  
	
	It is straightforward to show that  
	\[
	\left| \left\langle \partial_x^{s+1}(u^2),\ \partial_x^s u \right\rangle \right| \leq C(s) \sum_{m=0}^{s+1} \left| \int_{\mathbb{T}} \partial_x^m u\ \partial_x^{s+1-m} u\ \partial_x^s u\ dx \right|.
	\]  
	For \(m = 0\) or \(m = s+1\),
	using Sobolev inequalities,  
	\[
	\left| \int_{\mathbb{T}} u\ \partial_x^{s+1} u\ \partial_x^s u\ dx \right| = \frac{1}{2} \left| \int_{\mathbb{T}} \partial_x u\ (\partial_x^s u)^2\ dx \right| \leq C \|\partial_x u\|_{\infty} \|\partial_x^s u\|^2 \leq C \|u\|_2 \|\partial_x^s u\|^2.
	\]  
	For \(m = 1\) or \(m = s\), again, using Sobolev inequalities,  
	\[
	\left| \int_{\mathbb{T}} \partial_x u\ \partial_x^s u\ \partial_x^s u\ dx \right| \leq \|\partial_x u\|_{\infty} \|\partial_x^s u\|^2 \leq C \|u\|_2 \|\partial_x^s u\|^2.
	\]  
	For \(2 \leq m \leq s-1\), using Sobolev, Cauchy–Schwarz, and Poincaré inequalities
	\begin{align*}
		\left| \int_{\mathbb{T}}  \partial^m_x u \ \partial^{s + 1- m}_x u \ \partial^s_x u \ dx \right| &  \le \| \partial^m_x u \|_{{\infty}} \int_{\mathbb{T}} | \partial^{s + 1 - m}_x u| | \partial^s_x u | dx\\
		& \le  C \| u\|_{m + 1} \| \partial^{s + 1 - m}_x u \| \| \partial^s_x u\| \\
		& \le C \| u\|_s \| u\|_{s - 1} \| \partial^s_x u\| \\
		& \le C \Bigg( \|u\| + \| \partial_x^s u\| \Bigg) \| u\|_{s - 1} \| \partial^s_x u\|. 
	\end{align*}
	From the above, we deduce:  
	\[
	\left| \left\langle \partial_x^{s+1}(u^2),\ \partial_x^s u \right\rangle \right| \leq C \Big(\|u\|_2 \|\partial_x^s u\|^2 + \|u\| \|u\|_{s-1} \|\partial_x^s u\| + \|u\|_{s-1} \|\partial_x^s u\|^2\Big).
	\]  
	Similarly, we can show:  
	\[
	\left| \left\langle \partial_x^{s+1}(\zeta u),\ \partial_x^s u \right\rangle \right| \leq C \Big(\|\zeta\|_{s+2} \|u\| \|\partial_x^s u\| + \|\zeta\|_{s+2} \|\partial_x^s u\|^2\Big).
	\]  
	
	Using the induction hypothesis that \eqref{a_priori_est} holds for \(s-1\), we conclude:  
	\[
	\frac{d}{dt} \|\partial_x^s u\|^2 \leq C (1 + \|\varphi\|_s^2 + \|\partial_x^s u\|^2).
	\]  
	Applying Gronwall’s inequality as in the earlier steps, we obtain:  
	\[
	\|\partial_x^s u(t)\|^2 \leq C, \quad \forall t \in [0, T].
	\]  
	
	From \eqref{S=0_inequality}, \eqref{s=2_ineq}, and the above bound, we conclude:  
	\begin{align}\label{s_est}
		\sup_{t \in [0, T]} \| u(t) \|^2_s \le C, \quad \quad s \ge 3.
	\end{align}
	Thus, \eqref{a_priori_est} follows from \eqref{S=0_inequality}–\eqref{s_est}, and we establish the existence of a solution on \([0, T]\) for \eqref{xtd_equn}, i.e.,  
	the system \eqref{xtd_equn} has a solution \(u \in C([0, T]; H^s(\mathbb{T}))\).
	
	To establish uniqueness, let \( u_1, u_2 \in C([0, T]; H^s) \) be two solutions of \eqref{xtd_equn}. Define \( v = u_1 - u_2 \); then \( v \) satisfies the following equation:  
	\[
	\begin{cases}
		v_t - v_{xxxxx} + v_{xxx} + u_1 v_x + (u_2)_x v + (\zeta v)_x = 0, & \text{in } (0, T) \times \mathbb{T}, \\
		v(0) = 0, & \text{in } \mathbb{T}.
	\end{cases}
	\]  
	
	Multiplying this equation by \( 2v \) and integrating over \(\mathbb{T}\), we use the Sobolev and Cauchy–Schwarz inequalities to estimate the terms. This leads to:  
	\[
	\frac{d}{dt} \|v\|^2 \leq C \big(1 + \|\zeta\|_2 + \|u_1\|_2 + \|u_2\|_2\big) \|v\|^2.
	\]  
	
	By Gronwall's inequality, this implies \( v \equiv 0 \). Therefore, \( u_1 \equiv u_2 \), establishing the uniqueness of the solution.
\end{proof}

\subsection{Proof of Proposition \ref{prp_continty}}
\begin{proof}
	Let \( u = \mathcal{R}(u_0, 0, \varphi) \) and \( v = \mathcal{R}(v_0, 0, \varphi) \), with \([u_0] = [v_0] = a\). Define \(\Tilde{u} = u - a\) and \(\Tilde{v} = v - a\). Using \eqref{a_priori_est}, we have  
	\begin{align}\label{est_uv}
		\sup_{t \in [0, T]} \|\Tilde u(t)\|_s \leq C, \quad \sup_{t \in [0, T]} \|\Tilde v(t)\|_s \leq C,
	\end{align}  
	where \( C > 0 \) depends on \( T, s, \|u_0\|_s, \|v_0\|_s \), and \( \|\varphi\|_{L^2(0, T; H^s(\mathbb{T}))} \).
	
	For any \( n \in \mathbb{N}^* \), divide the interval $[0,T]$ into subintervals \( I_k = \left( \frac{kT}{n}, \frac{(k+1)T}{n} \right) \) for \( k = 0, 1, \dots, n-1 \). Now, we introduce the following mapping:  
	\[
	\Phi({z})(t) = \mathcal{W}(t) \Tilde u\left(\frac{kT}{n}\right) + \int_{\frac{kT}{n}}^t \mathcal{W}(t - s) \left( -{z}{z}_x + \varphi \right)(s) \, ds, \quad t \in I_k, \; k = 0, 1, \dots, n-1,
	\]  
	where \(\{\mathcal{W}(t)\}\) is the strongly continuous unitary group on \(L^2(\mathbb{T})\) generated by the linear operator \(\mathcal{A}\).
	
	Using a method similar to the proof of Proposition \ref{prp_existance}, for \({z}, {z}_1, {z}_2 \in Z_{\frac{1}{2}, s}(I_k) \cap L^2(I_k, L_0^2(\mathbb{T}))\), we can deduce from \eqref{est_uv} that  
	\begin{align}
		\| \Phi({z}) \|_{Z_{\frac{1}{2}, s}(I_k)} &\leq C \left( \left\| \Tilde u\left(\frac{kT}{n}\right) \right\|_s + \|\varphi\|_{L^2(0, T; H^s(\mathbb{T}))} \right) + C \left(\frac{T}{n}\right)^{\alpha} \|{z}\|^2_{Z_{\frac{1}{2}, s}(I_k)} \notag\\ 
		&\leq C_1 + C_2 \left(\frac{T}{n}\right)^{\alpha} \|{z}\|^2_{Z_{\frac{1}{2}, s}(I_k)}, \label{c1est}\\
		\| \Phi({z}_1) - \Phi({z}_2) \|_{Z_{\frac{1}{2}, s}(I_k)} &\leq C_3 \left(\frac{T}{n}\right)^{\alpha} \left( \|{z}_1\|_{Z_{\frac{1}{2}, s}(I_k)} + \|{z}_2\|_{Z_{\frac{1}{2}, s}(I_k)} \right) \|{z}_1 - {z}_2\|_{Z_{\frac{1}{2}, s}(I_k)},\label{c3est}
	\end{align}  
	where, \( C_1, C_2, \) and \( C_3 \) are constants independent of \( n \) and \( k \).
	
	Define  
	\begin{equation}\label{defR}
		R = 2C_1.
	\end{equation}
	
	For \( {z}, {z}_1, {z}_2 \in B_{Z_{\frac{1}{2}, s}(I_k)}(R) := \left\{ u \in Z_{\frac{1}{2}, s}(I_k) \cap L^2(I_k; L_0^2(\mathbb{T})) \; \big| \; \|u\|_{Z_{\frac{1}{2}, s}(I_k)} \leq R \right\}, \)  
	we select \( N_1 \in \mathbb{N}^* \) such that for \( n \geq N_1 \),
	\[
	\left(\frac{T}{n}\right)^{\alpha} R \max\{C_2, 2C_3\} \leq \frac{1}{2}.
	\]  
	Using \eqref{c1est}, \eqref{c3est}, and \eqref{defR}, this ensures that
	\[
	\|\Phi({z})\|_{Z_{\frac{1}{2}, s}(I_k)} \leq R,
	\]  
	\[
	\|\Phi({z}_1) - \Phi({z}_2)\|_{Z_{\frac{1}{2}, s}(I_k)} \leq \frac{1}{2} \|{z}_1 - {z}_2\|_{Z_{\frac{1}{2}, s}(I_k)}.
	\]  
	
	Thus, \( \Phi \) has a unique fixed point in \( B_{Z_{\frac{1}{2}, s}(I_k)}(R) \), which coincides to \( \Tilde u \) on \( I_k \times \mathbb{T} \). Furthermore,  
	\begin{equation}\label{esttu}
		\|\Tilde u\|_{Z_{\frac{1}{2}, s}(I_k)} \leq C, \quad \forall k = 0, 1, \dots, n-1, \quad n \geq N_1,
	\end{equation}  
	where \( C > 0 \) depends on \( T, s, \|u_0\|_s \), and \( \|\varphi\|_{L^2(0, T; H^s(\mathbb{T}))} \), but is independent of \( n \) and \( k \).  
	
	Similarly, we can find \( N_2 \in \mathbb{N}^* \) such that  
	\begin{equation}\label{esttv}
		\|\Tilde v\|_{Z_{\frac{1}{2}, s}(I_k)} \leq C, \quad \forall k = 0, 1, \dots, n-1, \quad n \geq N_2,
	\end{equation}  
	where \( C > 0 \) depends on \( T, s, \|v_0\|_s \), and \( \|\varphi\|_{L^2(0, T; H^s(\mathbb{T}))} \), but is independent of \( n \) and \( k \).
	
	Let $w=u-v=\tilde u-\tilde v$, then $w$ satisfies the following equation:
	\begin{align}\label{equw}
		\begin{cases}
			w_t - w_{xxxxx} + w_{xxx} + a w_x + \frac{1}{2}\left( \left( \tilde u+\tilde v\right)w \right)_x  = 0, & \text{in } (0, T) \times \mathbb{T}, \\
			w(0) = u_0-v_0, & \text{in } \mathbb{T}.
		\end{cases}
	\end{align}
	Additionally, $\left[ w\right]=0 $. Rewriting this equation in its integral form gives:
	
	\[
	w(t) = \mathcal{W}(t) w\left(\frac{kT}{n}\right) -\frac{1}{2} \int_{\frac{kT}{n}}^t \mathcal{W}(t - s) \left( \left( \tilde u+\tilde v\right)w \right)_x(s) \, ds, \quad t \in I_k, \; k = 0, 1, \dots, n-1.
	\]  
	Using estimates \eqref{esttu} and \eqref{esttv}, along with properties \eqref{semigrp_est} and \eqref{bilinear_est} from Lemma \ref{Bourgain_space}, we obtain the following bound for $k = 0, 1, \dots, n-1$ with $n\geq\max \left\lbrace N_1,N_2\right\rbrace $:
	
	\begin{align*}
		\| w \|_{Z_{\frac{1}{2}, s}(I_k)} &\leq C  \left\| w\left(\frac{kT}{n}\right) \right\|_s + C \left(\frac{T}{n}\right)^{\alpha}\left( \| \tilde u \|_{Z_{\frac{1}{2}, s}(I_k)}+\| \tilde v \|_{Z_{\frac{1}{2}, s}(I_k)}\right)  \|{w}\|_{Z_{\frac{1}{2}, s}(I_k)} \notag\\ 
		&\leq C_4 \left\| w\left(\frac{kT}{n}\right) \right\|_s+ C_5 \left(\frac{T}{n}\right)^{\alpha} \|{w}\|_{Z_{\frac{1}{2}, s}(I_k)},
	\end{align*}
	where \( C_5 \) depends on \( T, s, \|u_0\|_s, \|v_0\|_s \), and \( \|\varphi\|_{L^2(0, T; H^s(\mathbb{T}))} \), but is independent of \( n \) and \( k \).  We can find \( N_3 \in \mathbb{N}^* \) such that for \( n \geq N_3 \), 
	
	\[
	\left(\frac{T}{n}\right)^{\alpha} C_5 \leq \frac{1}{2}.
	\]
	
	This implies:
	
	\[
	\| w \|_{Z_{\frac{1}{2}, s}(I_k)} \leq 2 C_4 \left\| w\left(\frac{kT}{n}\right) \right\|_s.
	\]
	
	Let \( N = \max\{N_1, N_2, N_3\} \). By applying Lemma \ref{Bourgain_space}, for \( k = 0, 1, \dots, N-1 \), we obtain:
	
	\[
	\sup_{t \in \overline{I_k}} \|w(t)\|_s \leq C \| w \|_{Z_{\frac{1}{2}, s}(I_k)} \leq 2 C C_4 \left\| w\left(\frac{kT}{n}\right) \right\|_s = C_6 \left\| w\left(\frac{kT}{n}\right) \right\|_s,
	\]
	
	where \( C_6 > 1 \) is a constant depending on \( T, N, s \) but independent of \( k \). Repeating this process iteratively, we have:
	
	\[
	\sup_{t \in \overline{I_k}} \|w(t)\|_s \leq C_6 \left\| w\left(\frac{kT}{n}\right) \right\|_s \leq C_6 \sup_{t \in \overline{I_{k-1}}} \|w(t)\|_s \leq C_6^2 \left\| w\left(\frac{(k-1)T}{n}\right) \right\|_s,
	\]
	
	and so on. This leads to:
	
	\[
	\sup_{t \in \overline{I_k}} \|w(t)\|_s \leq C_6^{k+1} \|w(0)\|_s,
	\]
	
	for \( k = 0, 1, \dots, N-1 \). Consequently, we find:
	
	\[
	\sup_{t \in [0, T]} \|u(t) - v(t)\|_s \leq C_6^N \|u_0 - v_0\|_s.
	\]
	
	Thus, the Lipschitz property \eqref{lips_property} follows.
\end{proof}

\subsection{Proof of Proposition \ref{prp_limit}}
 \begin{proof}
	Let \(\delta > 0\). Recall the equation \eqref{xtd_dlta_equn}:  
	\begin{align}\label{asymp_dlta_equn}
		\begin{cases}
			u_t - (u + \delta^{-\frac{1}{2}} \zeta)_{xxxxx} + (u + \delta^{-\frac{1}{2}} \zeta)_{xxx} + (u + \delta^{-\frac{1}{2}} \zeta)(u + \delta^{-\frac{1}{2}} \zeta)_x = \delta^{-1} \eta, & \text{in } (0,T) \times \mathbb{T}, \\
			u(0, \cdot) = u_0(\cdot), & \text{in } \mathbb{T}.
		\end{cases}
	\end{align}  
	
	For \( u_0 \in H^{s + 7}(\mathbb{T}) \), \(\zeta \in H_0^{s + 8}(\mathbb{T}) \), and \(\eta \in H_0^{s + 7}(\mathbb{T}) \), Proposition \ref{prp_existance} ensures that for each \(\delta > 0\), equation \eqref{asymp_dlta_equn} has a unique solution \( u \in C([0, T]; H^s) \).  
	
	Next, we define the functions:  
	\[
	w(t) := u_0 + t(\eta - \zeta \zeta_x), \quad z(t) := u(\delta t) - w(t),
	\]  
	where \( t \in \left[0, \min\{1, T / \delta\}\right] \).  
	
	Substituting into the governing equation, \(z\) satisfies the following equation:  
	{\small 
		\begin{align}\label{z_equn}
			\begin{cases}
				z_t - \delta \left(  (z + w + \delta^{-\frac{1}{2}} \zeta)_{xxxxx} + (z + w + \delta^{-\frac{1}{2}} \zeta)_{xxx} + z(z + w + \delta^{-\frac{1}{2}} \zeta)_x + w(z + w + \delta^{-\frac{1}{2}} \zeta)_x \right)  + \delta^{\frac{1}{2}} \zeta(z + w)_x = 0, \\
				z(0) = 0.
			\end{cases}
	\end{align}  }
	We aim to show that  
	\[
	\sup_{t \in [0, \min\{1, T/\delta\}]} \| z(t) \|_s^2 \leq C \delta^{\frac{1}{2}}, \quad \text{for } s \in \mathbb{N}.
	\]  
	
	Case \(s = 0\):  
	Following the approach used in \eqref{S=0_inequality}, we take the scalar product of equation \eqref{z_equn} in \(L^2(\mathbb{T})\) with \(2z\). Using standard estimates, we obtain  
	
	\begin{align}\label{s=0_z_est}
		\sup_{t \in [0, \min\{1, T/\delta\}]} \| z(t) \|^2 \leq C \delta^{\frac{1}{2}}.  
	\end{align}  
	
	Case \(s = 2\):  
	Taking the scalar product of equation \eqref{z_equn} in \(L^2(\mathbb{T})\) with \(\mathcal{I} := z^2 + 2z_{xx} - 2z_{xxxx}\), and using the analysis from Proposition \ref{prp_existance}, along with \eqref{s=0_z_est}, we derive  
	\[
	\frac{d}{dt} \left( -\frac{1}{3} \int_{\mathbb{T}} z^3 \, dx + \| z_x(t) \|^2 + \| z_{xx}(t) \|^2 \right) \leq C \delta^{\frac{1}{2}} \Big( 1 + \| z_x \|^2 + \| z_{xx} \|^2 \Big).  
	\]  
	
	Using \( \delta < \delta^{\frac{1}{2}} \) for \( \delta \in (0, 1) \), integrating over \([0, t]\), and applying \eqref{s=0_z_est}, we obtain  
	\[
	\| z_x(t) \|^2 + \| z_{xx}(t) \|^2 \leq C \delta^{\frac{1}{2}} \Bigg( 1 + \int_0^t \Big( \| z_x(s) \|^2 + \| z_{xx}(s) \|^2 \Big) \, ds \Bigg).  
	\]  
	
	Applying Gronwall's inequality yields  
	\[
	\| z_x(t) \|^2 + \| z_{xx}(t) \|^2 \leq C \delta^{\frac{1}{2}}, \quad \forall t \in [0, \min\{1, T/\delta\}].  
	\]  
	
	Combining this with \eqref{s=0_z_est}, we establish  
	\begin{align}\label{s=2_z_est}
		\sup_{t \in [0, \min\{1, T/\delta\}]} \| z(t) \|_2^2 \leq C \delta^{\frac{1}{2}}.  
	\end{align}  
	
	Case \(s = 1\): Similar arguments involving the integration by parts and Sobolev embeddings lead to an intermediate estimate. Using interpolation between \(s = 0\) and \(s = 2\), it can be shown that  
	\begin{align}\label{s=1_z_est}
		\sup_{t \in [0, \min\{1, T/\delta\}]} \| z(t) \|_1^2 \leq C \delta^{\frac{1}{2}}.  
	\end{align}

	Case \(s \geq 3\):  
	Using induction and the estimates from \eqref{s=0_z_est}–\eqref{s=1_z_est}, we proceed as in Proposition \ref{prp_existance} to deduce that  
	\[
	\sup_{t \in [0, \min\{1, T/\delta\}]} \| z(t) \|_s^2 \leq C \delta^{\frac{1}{2}}, \quad \forall s \geq 3.  
	\]  
	
	Choose \(\delta_0 \in (0, 1)\) such that \(T / \delta_0 > 1\). For any \(\delta \in (0, \delta_0)\), we then have for all \(t \in [0, 1]\) and \(s \in  \mathbb{N}\):  
	\[
	\| z(t) \|_s^2 \leq C \delta^{\frac{1}{2}}.  
	\]  
	
	In particular,  
	\[
	\| z(1) \|_s^2 \leq C \delta^{\frac{1}{2}},  
	\]  
	which implies  
	\[
	\mathcal{R}_{\delta}(u_0, \delta^{-\frac{1}{2}} \zeta, \delta^{-1} \eta) \to u_0 + \eta - \zeta \zeta_x \quad \text{in } H^s(\mathbb{T}) \text{ as } \delta \to 0^+.  
	\]  
	This completes the proof.
\end{proof}
\section*{Acknowledgments}  

The authors extend their heartfelt gratitude to Shirshendu Chowdhury and Rajib Dutta for insightful discussions. Debanjit Mondal is deeply thankful for the support provided by the Integrated PhD Fellowship at IISER Kolkata. Sakil Ahamed acknowledges the financial assistance from the institute post-doctoral fellowship of IISER Kolkata. Additionally, we recognize the support of the FIST project (Project No. SR/FST/MS-II/2019/51) facilitated by the Department of Mathematics and Statistics, IISER Kolkata.
\bibliographystyle{plain}
\bibliography{reference}
\end{document}